\documentclass[11pt]{article}
\usepackage{amsfonts}
\usepackage{latexsym}
\usepackage{amsmath}
\usepackage{amssymb}
\usepackage{amsthm}
\usepackage{amsmath}
\usepackage{verbatim}
\usepackage{bibspacing}

\newtheorem{thm}{Theorem}[section]
\newtheorem*{thm*}{Theorem}
\newtheorem{cor}[thm]{Corollary}

\newtheorem{lem}[thm]{Lemma}

\newtheorem{prop}[thm]{Proposition}
\newtheorem*{prop*}{Proposition}

\newtheorem*{conj*}{Conjecture}
\newtheorem{defn}[thm]{Definition}
\newtheorem*{dfn*}{Definition}
\theoremstyle{definition}
\newtheorem{rem}[thm]{\textbf{Remark}}
\newtheorem*{rmk*}{Remark}
\newtheorem*{fact*}{Fact}

\theoremstyle{proof}

\newcommand{\norm}[1]{\left\Vert#1\right\Vert}

\newcommand{\abs}[1]{\left\vert#1\right\vert}
\newcommand{\set}[1]{\left\{#1\right\}}
\newcommand{\brac}[1]{\left(#1\right)}

\newcommand{\Real}{\mathbb{R}}

\newcommand{\Natural}{\mathbb{N}}
\newcommand{\C}{\mathcal{C}}

\newcommand{\F}{\mathcal{F}}
\newcommand{\eps}{\epsilon}

\numberwithin{equation}{section}

\begin{document}

\title{Complemented Brunn--Minkowski Inequalities and Isoperimetry for Homogeneous and Non-Homogeneous Measures}

\author{Emanuel Milman\textsuperscript{1} and Liran Rotem\textsuperscript{2}}
\date{}

\footnotetext[1]{Department of Mathematics, Technion - Israel
Institute of Technology, Haifa 32000, Israel. Supported by ISF (grant no. 900/10), BSF (grant no. 2010288), Marie-Curie Actions (grant no. PCIG10-GA-2011-304066) and the E. and J. Bishop Research Fund. Email: emilman@tx.technion.ac.il.}

\footnotetext[2]{Department of Mathematics, Tel-Aviv University. Supported by ISF (grant no. 826/13), BSF (grant no. 2012111) and by the Adams Fellowship Program of the Israel
Academy of Sciences and Humanities. Email: liranro1@post.tau.ac.il.}

\maketitle

\begin{abstract}
Elementary proofs of sharp isoperimetric inequalities on a normed space $(\Real^n,\norm{\cdot})$ equipped with a measure $\mu = w(x) dx$ so that $w^p$ is homogeneous are provided, along with a characterization of the corresponding equality cases. When $p \in (0,\infty]$ and in addition $w^p$ is assumed concave, the result is an immediate corollary of the Borell--Brascamp--Lieb extension of the classical Brunn--Minkowski inequality, providing a new elementary proof of a recent result of Cabr\'e--Ros Oton--Serra. When $p \in (-1/n,0)$, the relevant property turns out to be a novel ``$q$--complemented Brunn--Minkowski" inequality:
\begin{multline*} \forall \lambda \in (0,1) \;\;\;  \forall \text{ Borel sets $A,B \subset \Real^n$ such that } \mu(\Real^n \setminus A),\mu(\Real^n \setminus B) <\infty ~, \\
\mu_*(\Real^n \setminus (\lambda A + (1-\lambda) B)) \leq \brac{\lambda \mu(\Real^n \setminus A)^{q} + (1-\lambda) \mu(\Real^n \setminus B)^{q}}^{1/q} ~,
\end{multline*}
which we show is always satisfied by $\mu$ when $w^p$ is homogeneous with $\frac{1}{q} = \frac{1}{p} + n$; in particular, this is satisfied by the Lebesgue measure with $q = 1/n$. This gives rise to a new class of measures, which are ``complemented" analogues of the class of convex measures introduced by Borell, but which have vastly different properties. The resulting isoperimetric inequality and characterization of isoperimetric minimizers extends beyond the recent results of Ca\~{n}ete--Rosales and Howe. The isoperimetric and Brunn-Minkowski type inequalities also extend to the non-homogeneous setting, under a certain log-convexity assumption on the density. Finally, we obtain functional, Sobolev and Nash-type versions of the studied inequalities. 
\end{abstract}

\section{Introduction}

The well-known intimate relation between the classical isoperimetric inequality on Euclidean space on one hand, and the Brunn--Minkowski inequality for the volume of algebraic sums of sets on the other, dates back to the 19th century (see e.g. \cite{GardnerSurveyInBAMS,Schneider-Book} and the references therein). The contribution we wish to put forth in this work is that more general isoperimetric inequalities are equally intimately related to more general Brunn--Minkowski type inequalities, some of which have already been discovered, notably by Pr\'ekopa, Leindler, Borell, Brascamp and Lieb, and some of which seem new and previously unnoticed. It is the main goal of this work to put these ``complemented Brunn--Minkowski" inequalities into light, along with their associated class of ``complemented-concave" measures, which are the counterparts to the class of convex measures introduced by Borell in \cite{BorellConvexMeasures} (cf.  Brascamp--Lieb \cite{BrascampLiebPLandLambda1}). We begin with some definitions and background. 

Let $(\Omega,d)$ denote a separable metric space, and let $\mu$ denote a Borel measure on $(\Omega,d)$. 
The outer and inner Minkowski boundary measures $\mu^+_d(A)$ and $\mu^{-}_d(A)$ (respectively) of a Borel set $A \subset \Omega$ are defined as:
\[
\mu^+_d(A) := \liminf_{\eps \to 0} \frac{\mu(A^d_{\eps} \setminus A)}{\eps} ~, ~ \mu^-_d(A) := \mu^+_d(\Omega \setminus A) = \liminf_{\eps \to 0} \frac{\mu((\Omega \setminus A)^d_\eps \setminus (\Omega \setminus A))}{\eps} ~,
\]
where $A_\eps^d := \set{x \in \Omega ; \exists y \in A \;\; d(x,y) < \eps}$ is the open $\eps$-extension of $A$.  The isoperimetric problem on $(\Omega,d,\mu)$ consists of finding a sharp lower bound on $\mu^+_d(A)$ or $\mu^{-}_d(A)$ as a function of $\mu(A)$. Sets of given measure on which this lower bound is attained are called isoperimetric minimizers. 

On a linear space, Minkowski addition is defined as $A + B := \set{a + b \; ; \; a \in A ~,~ b \in B}$, and homothetic scaling is denoted by $t A = \set{t  a  \; ;\;  a \in A}$, $t \in \Real$. When the metric $d$ is given by a norm $\norm{\cdot}_K$ having open unit-ball $K$, note that $A^d_\eps = A + \eps K$; this connects the study of volumetric properties of Minkowski addition to the isoperimetric problem in the limiting regime when $\eps \rightarrow 0$. 
Consequently, we may extend our setup beyond the traditional case when $d$ and $\norm{\cdot}$ are genuine metrics and norms, respectively. For instance, throughout this work we dispense with the symmetry requirement in the terms ``metric", ``norm" and ``semi-norm". In other words, we do not require that a metric $d$ satisfy $d(x,y) = d(y,x)$, and only require that our norms and semi-norms be \emph{positively} homogeneous: $\norm{\lambda x} = \lambda \norm{x}$ for all $\lambda \geq 0$. Furthermore, we may consider the isoperimetric problem when $d$ is given by a semi-norm, or more generally, when it is induced by an arbitrary open set $B$ containing the origin, by defining $A_{\eps}^d := A + \eps B$ and thus accordingly the outer and inner boundary measures, which we then denote by $\mu^+_B$ and $\mu^-_B$, respectively.

\medskip

The starting point of our investigation in this work is the following:

\begin{thm}[Borell \cite{BorellConvexMeasures}, Brascamp--Lieb \cite{BrascampLiebPLandLambda1}] \label{thm:intro-BBL}
Let $p \in [-1/n,\infty]$ and $q \in [-\infty,1/n]$ satisfy:
\begin{equation} \label{eq:intro-pq}
 \frac{1}{q} = \frac{1}{p} + n ~.
\end{equation}
Let $w : \Real^n \rightarrow \Real_+$ denote  a measurable function which is $p$-concave, namely:
\begin{multline*}
\forall x,y \in \Real^n \;\;\;  w(x) w(y) > 0 \;\;\;  \forall \lambda \in (0,1) \;\;\; \Rightarrow \\
w(\lambda x + (1-\lambda) y) \geq \brac{\lambda w(x)^p + (1-\lambda) w(y)^p}^{1/p}  ~.
\end{multline*}
Then the measure $\mu = w(x) dx$ is $q$-concave, namely:
\begin{multline*} \forall A,B \subset \Real^n \;\;\; \mu(A) \mu(B) > 0 \;\;\;  \forall \lambda \in (0,1) \;\;\; \Rightarrow \\ 
\mu(\lambda A + (1-\lambda) B) \geq \brac{\lambda \mu(A)^q + (1-\lambda) \mu(B)^q}^{1/q} ~.
\end{multline*}
\end{thm}
Theorem \ref{thm:intro-BBL} is a generalization of the classical Brunn--Minkowski ($p=\infty$ and $w$ constant) and Prekop\'a--Leindler ($p=0$) Theorems, where in the above $\ell_p$ and $\ell_q$ averages are interpreted appropriately - see Section \ref{sec:Positive-P}. 
When $p> 0$ then $w$ is $p$-concave iff $w^p$ is concave on its convex support; the case $p=\infty$ corresponds to a constant density supported on a convex set;  $w$ is $0$-concave iff $\log w : \Real^n \rightarrow \Real \cup \set{-\infty}$ is concave ; and when $p<0$ then $w$ is $p$-concave iff $w^p : \Real^n \rightarrow \Real \cup \set{+\infty}$ is convex.  For further information, properties and generalizations of these well-studied densities and measures, which Borell dubs ``convex measures", we refer to e.g. \cite{BorellConvexMeasures,Borell-logconcave,BrascampLiebPLandLambda1,BobkovLedoux,BobkovConvexHeavyTailedMeasures} and to Section \ref{sec:Positive-P}. Throughout this work, we will assume that $p$ and $q$ are related by (\ref{eq:intro-pq}). 

\subsection{Homogeneous Densities - Positive Exponent}

Our first elementary observation is that when the density $w$  is in addition $p$-homogeneous, namely:
\[
w(\lambda x)^p = \lambda w(x)^p \;\;\; \forall \lambda > 0 ~,
\]
Theorem \ref{thm:intro-BBL} immediately yields the following corollary (the straightforward proof is deferred to Section \ref{sec:Positive-P}):

\begin{cor} \label{cor:intro-BBL}
Let $\mu = w(x) dx$, and assume that $w : \Real^n \rightarrow \Real_+$ is $p$-concave \emph{and} $p$-homogeneous \emph{with $p\in (0,\infty]$}. 
Let $\norm{\cdot}_K$ denote an arbitrary semi-norm on $\Real^n$ with open unit ball $K$. 
Then the following isoperimetric inequality holds:
\begin{equation} \label{eq:intro-isop-positive-p}
\mu^{+}_{\norm{\cdot}_K}(A) \geq \frac{1}{q} \mu(K)^q  \mu(A)^{1-q} \;\; \text{ with } \;\; \frac{1}{q} = \frac{1}{p} + n ~,
\end{equation}
for any Borel set $A$ with $\mu(A) < \infty$, and with equality for $A = t K$, $t > 0$. In other words, homothetic copies of $K$ are isoperimetric minimizers for $(\Real^n,\norm{\cdot}_K,\mu)$.
\end{cor}
Note that in such a case, $\mu$ is necessarily supported in a convex cone. 
The case that $p=\infty$, namely that $\mu$ has constant density in a convex cone $\Sigma$, with $K$ being a Euclidean ball, is due to Lions and Pacella \cite{LionsPacellaIsopOnCone}, who generalized the classical isoperimetric inequality for the Lebesgue measure.  Under some additional technical assumptions on the density $w$, the case of arbitrary $p>0$ and norm $\norm{\cdot}_K$ is due to Cabr\'e, Ros-Oton and Serra \cite{CabreRosOtonSerraCRAS,CabreRosOtonSerra}, who used a PDE approach for their proof; under some stronger technical assumptions and when $K$ is the Euclidean ball, this has also been subsequently verified by Ca\~{n}ete and Rosales \cite{CaneteRosales-IsopProblemForHomogeneousDensities}.  As described in \cite{CabreRosOtonSerra}, one of the present authors has noted that the case when $p=1/N$ and $N$ is an integer may be reduced to the $p=\infty$ case. The observation that the Lions--Pacella result follows from the Brunn--Minkowski Theorem was also noted by Barthe and Cordero--Erausquin (private communication), and the generalization to the Cabr\'e-et-al setting was also noted independently of our work by Nguyen (private communication). 

Using a result of Dubuc \cite{Dubuc1977EqualityInBorell} characterizing the equality cases in Theorem \ref{thm:intro-BBL}, it is not difficult to show that up to sets of zero $\mu$-measure, homothetic (and possibly translated) copies of $K$ are the unique \emph{convex}
 isoperimetric minimizers in Corollary \ref{cor:intro-BBL}, see Section \ref{sec:Positive-P} for more details. 
The uniqueness of general 
isoperimetric minimizers in the Lions--Pacella scenario in the case that $\Sigma \setminus \set{0}$ is smooth was obtained by Lions and Pacella themselves \cite{LionsPacellaIsopOnCone}, and using a different approach by Ritor\'e and Rosales \cite{RitoreRosalesMinimizersInEulideanCones}; the smoothness assumption has been recently removed by Figalli and Indrei \cite{FigalliIndrei-IsopStabilityInCones}. In the case that $K$ is a Euclidean ball and under some additional technical smoothness assumptions on the density $w$ and the cone $\Sigma$, the uniqueness of smooth compact stable hypersurfaces in the more general setting of Corollary \ref{cor:intro-BBL} has been recently obtained by Ca\~{n}ete and Rosales \cite{CaneteRosales-StableMinimizersForHomogeneousConcaveDensities}, extending the stability results from \cite{RitoreRosalesMinimizersInEulideanCones}; however, this does not directly yield the uniqueness of minimizing hypersurfaces since a-priori these may not be smooth nor compact - these issues are to be resolved in \cite{CaneteRosales-IsopProblemForHomogeneousDensities}. Furthermore, the smoothness conditions imposed in \cite{CaneteRosales-StableMinimizersForHomogeneousConcaveDensities,CaneteRosales-IsopProblemForHomogeneousDensities} prohibit using densities $w$ which are invariant under translations in a certain direction, and consequently Ca\~{n}ete and Rosales do not obtain any translations in their characterization of minimizers, even though in general translations might be necessary (see Remark \ref{rem:translations}). 
We refer to \cite{CabreRosOtonSerraCRAS,CabreRosOtonSerra,CaneteRosales-StableMinimizersForHomogeneousConcaveDensities} for further information on the previously known results when $p \in (0,\infty]$.

\subsection{Homogeneous Densities - Negative Exponent}

In this work, we will be more interested in the case when $q<0$ (equivalently $p \in (-1/n,0)$). Observe that in this case, the right-hand-side of (\ref{eq:intro-isop-positive-p}) becomes meaningless; this is not surprising, since in that range of values, a $p$-homogeneous density $w$ is non-integrable at the origin, and so homothetic copies of $K$ must have infinite $\mu$-measure, and thus cannot serve as isoperimetric minimizers. However, it is still plausible to conjecture that their \emph{complements}, which have the same (inner) boundary measure yet finite mass, might be the sought-after minimizers. It turns out that this is indeed the case, and that moreover, the requirement that $w$ be $p$-concave is in fact not needed:

\begin{thm} \label{thm:intro-main1}
Let $\mu = w(x) dx$ where $w : \Real^n \setminus \set{0} \rightarrow \Real_+$ is a $p$-homogeneous Borel density with $p \in (-1/n,0)$. Let $\norm{\cdot}_K$ 
denote an arbitrary semi-norm on $\Real^n$ with open unit ball $K$ so that $\mu(\Real^n \setminus K) < \infty$. Then the following 
isoperimetric inequality holds:
\begin{equation} \label{eq:hom-iso}
\mu^{-}_{\norm{\cdot}_K}(C) \geq -\frac{1}{q} \mu(\Real^n \setminus K)^q  \mu(C)^{1-q} \;\; \text{ with } \;\; \frac{1}{q} = \frac{1}{p} + n ~,
\end{equation}
for any Borel set $C$ with $\mu(C) < \infty$, and with equality for $C = \Real^n \setminus tK$,
 $t > 0$.
In other words, complements of homothetic copies of $K$ are isoperimetric minimizers for $(\Real^n, \norm{\cdot}_K, \mu)$.
\end{thm}

In fact, the result remains valid when $\mu$ is an arbitrary $\sigma$-finite $q$-homogeneous Borel measure, satisfying $\mu^q(\lambda A) = \lambda \mu^q(A)$ for all Borel sets $A \subset \Real^n$ and $\lambda > 0$. Furthermore, we show in Section \ref{sec:Homogeneous} that up to sets of zero $\mu$-measure, complements of homothetic copies of $K$ are the \emph{unique} isoperimetric minimizers in Theorem \ref{thm:intro-main1} (no translations are necessary in the negatively homogeneous case). When $w$ is in addition assumed $p$-concave, under some additional technical smoothness assumptions on $w$ and the support of $\mu$, and when $K$ is the Euclidean ball, Theorem \ref{thm:intro-main1} and the uniqueness of isoperimetric minimizers has been established by Ca\~{n}ete and Rosales \cite{CaneteRosales-StableMinimizersForHomogeneousConcaveDensities,CaneteRosales-IsopProblemForHomogeneousDensities}.
When $w(x) = \abs{x}^\frac{1}{p}$ and $K$ is the Euclidean ball, Theorem \ref{thm:intro-main1} was previously obtained by D{\'{\i}}az, Harman, Howe and Thompson \cite{DHHT-IsopInSectorsWithDensity}. A more general result has been established by Howe \cite{Howe-IsoperimetryInWarpedProducts}, see Theorem \ref{thm:intro-main2} below. 

The proof of Theorem \ref{thm:intro-main1} is completely analogous to the one of Corollary \ref{cor:intro-BBL}, once it has been shown that any homogeneous measure $\mu$ as above is in fact \emph{complemented-concave}:

\begin{dfn*}
A Borel measure $\mu$ on $\Real^n$ is said to be \emph{$q$-complemented concave} (``$q$-CC"), and is said to satisfy a $q$-complemented Brunn--Minkowski inequality (``$q$-CBM"), $q \in [-\infty,+\infty]$, if:
\begin{multline} \label{def:intro-q-CBM}
\forall \lambda \in (0,1) \;\;\;  \forall \text{ Borel sets $A,B \subset \Real^n$ such that } \mu(\Real^n \setminus A),\mu(\Real^n \setminus B) <\infty ~, \\
\mu_*(\Real^n \setminus (\lambda A + (1-\lambda) B)) \leq \brac{\lambda \mu(\Real^n \setminus A)^{q} + (1-\lambda) \mu(\Real^n \setminus B)^{q}}^{1/q} ~.
\end{multline}
Here $\mu_*$ denotes the inner measure induced by $\mu$.
\end{dfn*}

The complemented Brunn--Minkowski  inequality considered above seems new, and constitutes the main impetus for this work. It is a-priori not even clear that there are non-trivial examples of $q$-complemented concave measures (at least for $q < \infty$), but it turns out that when $p \in (-\infty,-\frac{1}{n-1}] \cup (-\frac{1}{n},0) \cup (0,\infty]$ (equivalently $q \leq 1$ and $q \neq 0$), any $p$-homogeneous Borel density (and more generally, $\sigma$-finite $q$-homogeneous Borel measure) gives rise to such a measure. In particular, (\ref{def:intro-q-CBM}) holds with $q=1/n$ for $\mu$ being the Lebesgue measure restricted to an arbitrary measurable cone. Various additional properties satisfied by the class of $q$-CC measures are studied in Section \ref{sec:Properties}. One crucial feature of the $q$-CC class is that it is a convex cone when $q \leq 1$, much in contrast with the Borell--Brascamp--Lieb class of $q$-concave measures; other dissimilarities are also considered. In particular, we show that when $q < \infty$, any non-zero $q$-CC measure must have infinite mass, and that contrary to the Borell--Brascamp--Lieb Theorem \ref{thm:intro-BBL}, this class does not admit any local characterization. 

\subsection{Non-Homogeneous Log-Convex Densities} 

In Section \ref{sec:NonHomogeneous}, we go beyond the homogeneous setting, and obtain the following generalization of Theorem \ref{thm:intro-main1} (in fact, we obtain a more general version):

\begin{thm} \label{thm:intro-main2}
Let $\norm{\cdot}$ denote a norm on $\Real^n$ with open unit-ball $K$. Let $w_0 : \Real^n \rightarrow \Real_+$ denote a $1$-homogeneous measurable function so that $\int_K w_0(x) dx < \infty$. Let $\varphi : (0,\infty) \rightarrow \Real_+$ denote a non-increasing function, non-integrable at the origin, integrable at infinity, and so that $\log \varphi : (0,\infty) \rightarrow [-\infty,\infty)$ is convex. Set:
\[
\mu = w_0(x / \norm{x}_K) \norm{x}_K^{1-n} \varphi(\norm{x}_K) dx ~. 
\]
Then, denoting  $\Phi(t) := \int_t^\infty \varphi(s) ds$ and $I := \varphi \circ \Phi^{-1}$, the following isoperimetric inequality holds:
\[
\mu^{-}_{\norm{\cdot}_K}(C) \geq I(\mu(C)) 
\]
for any Borel set $C \subset \Real^n$ with $\mu(C) < \infty$, and with equality for $C = \Real^n \setminus tK$, $t > 0$.
In other words, complements of homothetic copies of $K$ are isoperimetric minimizers for $(\Real^n, \norm{\cdot}_K, \mu)$.
\end{thm}

Furthermore, we show in Section \ref{sec:NonHomogeneous} that when $\log \varphi$ is strictly convex or when the support of $\mu$ is convex, complements of homothetic copies of $K$ are, up to sets of zero $\mu$-measure, the \emph{unique} isoperimetric minimizers in Theorem \ref{thm:intro-main2} (among all sets whose complements have finite Lebesgue measure). Theorem \ref{thm:intro-main2} and the corresponding uniqueness of isoperimetric minimizers has been explicitly obtained by Howe \cite{Howe-IsoperimetryInWarpedProducts} for the case when $K$ is the Euclidean ball and $w_0$ and $\varphi$ are continuous and positive on $S^{n-1}$ and $(0,\infty)$, respectively, by using a warped-product representation of the space; see \cite[Section 3]{Howe-IsoperimetryInWarpedProducts} for additional isoperimetric results in more general situations. As Howe points out, his results and methods should apply for arbitrary star-shaped smooth compact domains $K$ (in the spirit of our more general formulation in Section \ref{sec:NonHomogeneous}), but even in that case, it seems that Howe's Riemannian warped-product boundary measure does not coincide with our ``Finslerian" notion, and there are various technicalities which should be taken care of, such as using the convexity of $K$ where needed and defining the corresponding measure $\mu$ correctly; furthermore, it is not clear how to extend Howe's  method to the case when $w_0$ is no longer assumed continuous or positive. Contrary to the PDE and Geometric Measure Theory approaches of Cabr\'e--Ros-Oton--Serra and Ca\~{n}ete--Rosales mentioned above, Howe's approach is based on the convexity of the function $I$ above and in that sense is essentially identical to our own approach. However, our emphasis is on the associated Brunn--Minkowski type inequality satisfied by measures $\mu$ as above:
\[
 \mu(\Real^n\setminus (A + t K)) \leq \Phi( \Phi^{-1}(\mu(\Real^n \setminus A)) + t \Phi^{-1}(\mu(\Real^n \setminus K)) ) ~,
 \]
 for all $t \geq 0$ and any Borel set $A$ so that $\mu(\Real^n \setminus A) < \infty$.
 
 \subsection{Functional Versions}
 
All of the isoperimetric inequalities formulated above also have global Brunn--Minkowski type versions, which are obtained in the corresponding sections. 
To complete the picture, we formulate and prove in Section \ref{sec:Functional} various additional functional versions of these inequalities. In particular, we obtain seemingly new Sobolev and Nash-type inequalities with respect to locally-integrable $p$-homogeneous densities on $\Real^n \setminus \set{0}$ for $ p \in (-1/n,0)$. For instance, we show that when $\mu$ is a measure on $\Real^n$ having such a density, then for any locally Lipschitz function $f : \Real^n \rightarrow \Real$ with $f(0) = 0$ and any norm $\norm{\cdot}$ on $\Real^n$, we have:
\[
\norm{f}_{L^1(\mu)} \leq C_2  (\int \norm{\nabla f}^* d\mu)^{\alpha}  \norm{f}_{L^\infty(\mu)}^{1-\alpha} ~,~ 
\]
and for any $\beta \in (0,\alpha)$:
\[
\int \norm{\nabla f}^* d\mu \geq C_1\brac{\frac{\alpha-\beta}{\alpha}}^{1/\beta} \norm{f}_{L^{\beta}(\mu)} ~,
\]
where:
\[
 \frac{1}{q} = \frac{1}{p} + n ~,~ \alpha := \frac{1}{1-q} ~,~  C_1 := -\frac{1}{q}  \mu(\set{\norm{x} \geq 1})^q ~,~ C_2 := C_1^{-\alpha} ~.
 \]

\bigskip
\begin{rem}
Most of the previously obtained isoperimetric results mentioned above employ the notion of weighted (De Giorgi) perimeter instead of our preferred notion of (Minkowski's) boundary measure (see e.g. \cite{BuragoZalgallerBook,MorganBook4Ed} for background and definitions). One may show (see \cite[Theorem 14.2.1]{BuragoZalgallerBook}) that the former notion is always a-priori stronger than the latter; however, for sufficiently regular sets, the two notions always coincide, and since isoperimetric minimizers are typically ensured to be ``sufficiently regular" by the powerful results of Geometric Measure Theory (see \cite{MorganBook4Ed}), a-posteriori the particular definition used makes no difference in many situations. Our preference to work with Minkowski's boundary measure is due to its intimate connection to Brunn--Minkowski-type inequalities, which are really the protagonists of this work; furthermore, it allows us to treat arbitrary Borel sets and measures, without any assumptions on smoothness. 
\end{rem}
 
\medskip

\textbf{Acknowledgement.} We thank Franck Barthe, Xavier Cabr\'e, Antonio Ca\~{n}ete, Dario Cordero-Erausquin, Sean Howe, Erwin Lutwak, Frank Morgan and C\'esar Rosales for helpful correspondence during the preparation of this work. We also thank the referee for carefully reading the manuscript and for his / her meticulous comments.

\section{Positively Homogeneous and Concave Measures} \label{sec:Positive-P}

\subsection{$p$-concave densities and $q$-concave measures}

We start with the following definitions. Let $E$ be an affine vector space. 
\begin{defn}
A measurable function $w : E \rightarrow \Real_+$ is called $p$-concave ($p \in [-\infty,\infty]$) if:
\[
\forall x,y \in E \;\; \text{ so that } \;\; w(x) w(y) > 0  ~,
\]
we have:
\[
\forall \lambda \in (0,1) \;\;\; w(\lambda x + (1-\lambda) y) \geq \brac{\lambda w(x)^p + (1-\lambda) w(y)^p}^{1/p} ~.
\]
Here and elsewhere the right-hand side is interpreted as $\max(w(x),w(y))$ when $p=\infty$, as $\min(w(x),w(y))$ when $p=-\infty$, and as $w(x)^\lambda w(y)^{1-\lambda}$ when $p=0$. 
\end{defn}
\begin{defn} \label{def:q-concave}
A Borel measure $\mu$ on $E$ is called $q$-concave ($q \in [-\infty,+\infty]$) if:
\[
\forall \text{ Borel subsets } A,B \subset E \;\; \text{ so that } \;\; \mu(A) \mu(B) > 0 ~,
\]
we have:
\begin{equation} \label{eq:BBL}
\forall \lambda \in (0,1) \;\;\;  \mu(\lambda A + (1-\lambda) B) \geq \brac{\lambda \mu(A)^q + (1-\lambda) \mu(B)^q}^{1/q}  ~.
\end{equation}
The cases $q \in \set{-\infty,0,\infty}$ are interpreted as above. 
\end{defn}

\begin{rem}
These definitions appear implicitly in the work of Brascamp and Lieb \cite{BrascampLiebPLandLambda1} and explicitly in the work of Borell \cite{BorellConvexMeasures}, who used the name ``convex measures" to encompass the entire class of $(-\infty)$-concave measures. The requirement that the above conditions should only hold when $w(x) w(y) > 0$ or $\mu(A) \mu(B) > 0$ is redundant when $p<0$ or $q<0$, respectively, but is otherwise crucial. The measurability of $\lambda A + (1-\lambda)B$ is addressed in Remark \ref{rem:analytic}. 
\end{rem}

The starting point of our investigation in this work is the following:
\begin{thm} \label{thm:BBL}
Let $p \in [-1/n,\infty]$ and $q \in [-\infty,1/n]$ satisfy:
\begin{equation} \label{eq:pq}
 \frac{1}{q} = \frac{1}{p} + n ~.
 \end{equation}
\begin{enumerate}
\item
(Borell \cite{BorellConvexMeasures}, Brascamp--Lieb \cite{BrascampLiebPLandLambda1}). 
If $w : \Real^n \rightarrow \Real_+$  is a $p$-concave function, then the measure $\mu = w(x) dx$ is $q$-concave.
\item
(Borell \cite[Theorem 3.1]{BorellConvexMeasures}). 
Conversely, if $\mu$ is a $q$-concave measure on $\Real^m$, then $\mu$ is supported on an $n$-dimensional affine subspace $E$, and has a density $w$ with respect to the Lebesgue measure on $E$ which is $p$-concave.
\end{enumerate} 
\end{thm}

\begin{rem} \label{rem:gen-BBL}
As already mentioned in the Introduction, the first part of the theorem is a generalization of the classical Brunn--Minkowski ($p=\infty$ and $w$ constant) and Prekop\'a--Leindler ($p=0$) Theorems (see e.g. \cite{GardnerSurveyInBAMS}). The restriction on the corresponding ranges of $p$ and $q$ above is due to the fact that there are no non-zero absolutely continuous $q$-concave measures when $q > 1/n$ \cite{BorellConvexMeasures}. By Jensen's inequality, we immediately see that the class of $p$-concave densities ($q$-concave measures) on $\Real^n$ becomes larger as $p$ ($q$) decreases from $+\infty$ to $-1/n$ ($1/n$ to $-\infty$). 
\end{rem}

We will assume throughout the rest of this section that $p$ and $q$ are related by (\ref{eq:pq}). For convenience, we note that $(-\infty,-1/n)$, $(-1/n,0)$ and $(0,\infty)$ in the $p$-domain get mapped to $(1/n,\infty)$, $(-\infty,0)$ and $(0,1/n)$ in the $q$-domain, respectively. 

\subsection{$p$-homogeneous densities and $q$-homogeneous measures}
 
 We will employ throughout this work the following notation. We denote by $B(x,r)$ the closed Euclidean ball centered at $x \in \Real^n$ and having radius $r>0$.  Furthermore:
 
\begin{defn}
A measurable function $w : \Real^n \rightarrow \Real_+$ is called $p$-homogeneous, $p \in (-\infty,\infty] \setminus \set{0}$, if:
\[
\forall x \in \Real^n \;\;\; \forall \lambda > 0 \;\;\; w(\lambda x)^p = \lambda w(x)^p ~,
\]
with the interpretation when $p=\infty$ that $w(\lambda x) = w(x)$ for all $\lambda > 0$, i.e. that $w$ is constant along rays from the origin.
\end{defn}
\begin{defn}
A Borel measure $\mu$ on $\Real^n$ is called $q$-homogeneous, $q \in (-\infty,\infty] \setminus \set{0}$, if:
\[
\forall \text{ Borel subset } A \subset \Real^n \;\;\; \forall \lambda > 0 \;\;\; \mu(\lambda A)^q = \lambda \mu(A)^q ~,
\]
with the interpretation when $q=\infty$ that $\mu(\lambda A) = \mu(A)$ for all $\lambda > 0$. 
\end{defn}

Notice that our $p$-homogeneous functions are often called $1/p$-homogeneous in the literature, and similarly for our $q$-homogeneous measures. Our convention will be more convenient in the sequel.

The following lemma is a trivial analogue of Theorem \ref{thm:BBL} for the class of homogeneous densities and measures:
\begin{lem} \label{lem:homogeneous}
Let $p \in (-\infty,\infty] \setminus \set{0}$ and  $q \in (-\infty,\infty] \setminus \set{0}$ be related by (\ref{eq:pq}). 
\begin{enumerate}
\item
Let $w : \Real^n \rightarrow \Real_+$  denote a $p$-homogeneous function. Then the measure $\mu = w(x) dx$ is $q$-homogeneous.
\item
Conversely, if $\mu = w(x) dx$ is a $q$-homogeneous measure on $\Real^n$, then $w$ coincides almost-everywhere with a $p$-homogeneous function. 
\end{enumerate} 
\end{lem}
\begin{proof}[Sketch of proof]
\begin{enumerate}
\item
\[
\mu(\lambda A) = \int_{\lambda A} w(x) dx = \lambda^n \int_{A} w(\lambda y) dy = \lambda^{n + 1/p} \int_A w(y) dy = \lambda^{1/q} \mu(A) ~.
\]
\item
This follows similarly by testing $\mu$'s homogeneity property on Euclidean balls of the form $B(x,\eps)$ with $\eps \rightarrow 0$ and applying Lebesgue's differentiation theorem. 
\end{enumerate}
\end{proof}

\subsection{$q$-concave and $q$-homogeneous measures with $q > 0$}

The following elementary observation encapsulates the usefulness of combining the two previously described properties:
\begin{prop} \label{prop:positive-p}
Let $\mu$ denote a $q$-concave \emph{and} $q$-homogeneous measure on $\Real^n$ \emph{with $q \in  (0,1/n]$}. Let $\norm{\cdot}_K$ denote an arbitrary semi-norm on $\Real^n$ with open unit ball $K$. 
Then the following isoperimetric inequality holds:
\begin{equation} \label{eq:isop-positive-p}
\mu^{+}_{\norm{\cdot}_K}(A) \geq \frac{1}{q} \mu(K)^q  \mu(A)^{1-q} ~,
\end{equation}
for any Borel set $A$ with $\mu(A) < \infty$, and with equality for $A = t K$, $t > 0$. In other words, homothetic copies of $K$ are isoperimetric minimizers for $(\Real^n,\norm{\cdot}_K,\mu)$.
\end{prop}

As an immediate corollary of the Borell--Brascamp--Lieb Theorem \ref{thm:BBL} and Lemma \ref{lem:homogeneous}, we obtain Corollary \ref{cor:intro-BBL} from the Introduction, which we state here again for convenience. See the remarks and references following Corollary \ref{cor:intro-BBL} for previously known results. 

\begin{cor}
Let $\mu = w(x) dx$, and assume that $w : \Real^n \rightarrow \Real_+$ is $p$-concave \emph{and} $p$-homogeneous \emph{with $p\in (0,\infty]$}. Then homothetic copies of $K$ are isoperimetric minimizers for $(\Real^n,\norm{\cdot}_K,\mu)$, and (\ref{eq:isop-positive-p}) holds with $\frac{1}{q} = \frac{1}{p} + n$. 
\end{cor}

\begin{proof}[Proof of Proposition \ref{prop:positive-p}]
We may assume that $\mu(A) \mu(K) > 0$ since otherwise there is nothing to prove. 
By scaling the $q$-concavity property (\ref{eq:BBL}) and employing the $q$-homogeneity, we obtain for any Borel sets $A,B \subset \Real^n$ with $\mu(A) \mu(B) > 0$:
\[
\mu(A + t B) \geq \brac{\mu(A)^q + t \mu(B)^q}^{1/q} \;\;\; \forall t > 0 ~.
\]
Applying this to $B = K$ and using that $\mu(A) < \infty$, we calculate the boundary measure of $A$:
\begin{eqnarray*}
\mu^+_{\norm{\cdot}_K}(A)  & = & \liminf_{\eps \rightarrow 0}\frac{\mu((A+\eps K) \setminus A)}{\eps}  = \liminf_{\eps \rightarrow 0} \frac{\mu(A + \eps K) - \mu(A) }{\eps} \\
& \geq  & \liminf_{\eps \rightarrow 0} \frac{\brac{\mu(A)^q + \eps \mu(K)^q}^{1/q} - \mu(A) }{\eps} = \frac{1}{q} \mu(A)^{1-q} \mu(K)^q ~,
\end{eqnarray*}
with equality for $A=K$, since $K + \eps K = (1+\eps) K$ when $K$ is convex. By scaling, we see that equality holds for all positive homothetic copies of $K$.  
\end{proof}

\begin{rem} \label{rem:no-convex}
Note that the convexity of $K$ was only used to assert that homothetic copies of $K$ are isoperimetric minimizers. The inequality (\ref{eq:isop-positive-p}) remains valid for arbitrary Borel sets $K$, but when $K$ is convex, this inequality acquires isoperimetric content on an appropriate metric-measure space (in fact, it is only a non-symmetric semi-metric). 
\end{rem}

\begin{rem} \label{rem:one-sided}
Inspecting the proof, note that we only used one-sided homogeneity - we just need that:
\[
\forall \eps \in [0,1] \;\;\; \mu(\eps A)^q \leq \eps \mu(A)^q ~,
\]
which we have if:
\[
\forall \eps \in [0,1] \;\;\; w(\eps x)^p \leq \eps w(x)^p ~.
\]
However, the concavity of $w^p$ implies that:
\[
\forall \eps \in [0,1] \;\;\; w(\eps x)^p - w(0)^p \geq \eps (w(x)^p - w(0)^p) ~,
\]
so this is only consistent if $w(0)=0$ and hence $w^p$ must be homogeneous. 
\end{rem}

\subsection{Generalizations and Characterization of Equality Cases}

In order to characterize the equality case in (\ref{eq:isop-positive-p}), we will need part (\ref{enum:dummy2}) of the following theorem. Part (\ref{enum:dummy1}) generalizes part (\ref{enum:dummy1}) of Theorem \ref{thm:BBL}.

\begin{thm} \label{thm:Dubuc}
Let $p \in [-1/n,\infty]$, $\lambda \in (0,1)$ and let $f,g,h : \Real^n \rightarrow \Real_+$ denote three integrable functions so that:
\[
h(\lambda x + (1-\lambda) y) \geq \brac{\lambda f(x)^p + (1-\lambda) g(y)^p}^{1/p} \;\;\; \text{ for almost all } (x,y) \in \Real^n \times \Real^n ~.
\]
Then:
\begin{enumerate}
\item (Borell \cite{BorellConvexMeasures}, Brascamp--Lieb \cite{BrascampLiebPLandLambda1}) \label{enum:dummy1}
\begin{equation} \label{eq:gen-BBL}
\int h \geq \brac{\lambda (\int f)^q + (1-\lambda) (\int g)^q}^{1/q} ~,
\end{equation}
with $\frac{1}{q} = \frac{1}{p} + n$. 
\item (Dubuc \cite[Theorem 12]{Dubuc1977EqualityInBorell})  \label{enum:dummy2}
If $\int f , \int g, \int h > 0$ and we have equality in (\ref{eq:gen-BBL}), then there exist a scalar $m_0 > 0$ and a vector $b_0 \in \Real^n$ so that:
\[
\frac{f(x/\lambda)}{\lambda^n \int f} = \frac{m_0^n g((m_0 x + b_0)/(1-\lambda))}{(1-\lambda)^n \int g} = \frac{(m_0+1)^n h((m_0+1) x + b_0)}{\int h} ~,
\]
for almost all $x \in \Real^n$.
\end{enumerate}
\end{thm}

\begin{rem}
An even more general formulation was obtained by Borell and Dubuc, whereby the $\ell_p$-norm in the assumption is replaced by an arbitrary $1$-homogeneous function $P : \Real_+^2 \rightarrow \Real_+$, in which case the $\ell_q$-norm in the conclusion should be replaced by: 
\[
P_n(u,v) := \inf_{t \in (0,1)} P(u t^{-n} , v (1-t)^{-n}) ~,~ (u,v) \in \Real_+^2 ~.
\]
\end{rem}

In addition, the following result of Dubuc, which is a particular case of \cite[Theorem 6]{Dubuc1977EqualityInBorell}, will be very useful for characterizing cases of equality in the sequel:
\begin{thm}[Dubuc] \label{thm:Dubuc-lem}
Let $C_1,C_2 \subset \Real^n$ be two Borel sets, and assume that $(\alpha C_1 + (1-\alpha) C_2) \setminus (C_1 \cap C_2)$ is a null-set. Then necessarily $C_1$ and $C_2$ coincide up to null sets with a common open convex set $C_3$. 
\end{thm}

\begin{cor} \label{cor:equality-global}
Let $w : \Real^n \rightarrow \Real_+$ denote a $p$-concave density, $p \in [-1/n,\infty]$, set $\frac{1}{q} = \frac{1}{p} + n$, and let $\mu = w(x) dx$ be the corresponding $q$-concave measure. Assume that:
\[
\mu(\lambda A+(1-\lambda) B) = \brac{\lambda \mu(A)^q + (1-\lambda) \mu(B)^q}^{1/q} ~,
\]
for two Borel sets $A,B \subset \Real^n$ with $\mu(A),\mu(B) \in (0,\infty)$ and $\lambda \in (0,1)$. Then there exist $m > 0$ and $b \in \Real^n$ so that $B$ coincides with $m A + b$ up to zero $\mu$-measure, and $w(m x + b) = c w(x)$ for almost all $x \in A$ and some $c > 0$. \\
Furthermore, if $w$ is $p$-homogeneous and $b=0$, then up to null-sets, $A \cap \Sigma$ and $B \cap \Sigma$ must be convex, where $\Sigma$ is the convex cone on which $\mu$ is supported.
\end{cor}
\begin{proof}
Use $f(x) = 1_A(x) w(x)$, $g(y) = 1_B(y) w(y)$ and $h(z) = 1_{\lambda A + (1-\lambda) B}(z) w(z)$ in Theorem \ref{thm:Dubuc} (2), and set $m = \frac{\lambda}{1-\lambda} m_0$ and $b = \frac{1}{1-\lambda} b_0$. 

For the second part, denote $A' := m A$, and observe by homogeneity that:
\[
\mu(\lambda' A'+(1-\lambda') B) = \brac{\lambda' \mu(A')^q + (1-\lambda') \mu(B)^q}^{1/q} = \mu(A') = \mu(B) = \mu(A' \cap B) ~,~ \lambda' := \frac{\lambda / m}{\lambda/m + (1-\lambda)} ~.
\]
In particular, $(\lambda' A' + (1-\lambda') B) \cap \Sigma$ coincides with $A' \cap B \cap \Sigma$ up to a null-set. Since $\Sigma$ is convex:
\[
(\lambda' A' + (1-\lambda') B) \cap \Sigma \supset \lambda' (A' \cap \Sigma) + (1-\lambda') (B \cap \Sigma) \supset A' \cap B \cap \Sigma ~,
\]
and so $\lambda' (A' \cap \Sigma) + (1-\lambda') (B \cap \Sigma)$ coincides with $A' \cap B \cap \Sigma$ up to a null-set. Invoking Theorem \ref{thm:Dubuc-lem}, it follows that $A' \cap \Sigma$ and $B \cap \Sigma$ must be convex up to null-sets, and the assertion follows. 
\end{proof}

The equality case on the infinitesimal level requires more justification. The surprising idea below, which may be traced to Bonnesen (see \cite[Section 49]{BonnesenFenchelBook}), is to reduce the problem to the equality case in the $q$-Brunn--Minkowski inequality.

\begin{cor} \label{cor:equality-local}
Let $\mu = w(x) dx$, and assume that $w : \Real^n \rightarrow \Real_+$ is $p$-concave \emph{and} $p$-homogeneous \emph{with $p \in (0,\infty]$}. Let
$A,B \subset \Real^n$ denote two
 convex sets\footnote{In the published version, $A$ and $B$ were erroneously assumed to only be Borel sets} with $\mu(A),\mu(B) \in (0,\infty)$. 
 Assume that:
\[
\mu^{+}_{B}(A) = \frac{1}{q} \mu(B)^q  \mu(A)^{1-q} ~,~ \frac{1}{q} = \frac{1}{p} + n ~.
\]
Then $A$ coincides with $m B + b$ up to zero $\mu$-measure, for some $m > 0$ and $b \in \Real^n$, and $w(m x + b) = c w(x)$ for almost all $x \in B$ and some $c > 0$. \\
\end{cor}
\begin{proof}
Set: 
\[
\Psi(t) := \mu((1-t) A + t B)^q = (1-t) \mu\brac{A + \frac{t}{1-t} B}^q ~,~ t \in [0,1] ~,
\]
and observe that $\Psi$ is concave by Theorem \ref{thm:BBL}
and the fact that $\lambda C + (1-\lambda) C = C$ for any convex $C$ and $\lambda \in [0,1]$. Consequently, its right derivative at $0$ exists (possibly equaling $+\infty$) and we have:
\[
\mu(B)^q - \mu(A)^q = \Psi(1) - \Psi(0) \leq \Psi'(0) = -\mu(A)^q + q \mu(A)^{q-1} \mu^+_{B}(A) ~,
\] 
which yields the isoperimetric inequality (\ref{eq:isop-positive-p}). Now assume we have equality in the isoperimetric inequality - this means that $\Psi'(0) = \Psi(1) - \Psi(0)$, and consequently $\Psi$ must be affine. In particular:
\[
\mu(\frac{1}{2} A + \frac{1}{2} B)^q = \Psi(1/2) = \frac{1}{2} \Psi(0) + \frac{1}{2} \Psi(1) = \frac{1}{2} \mu(A)^q + \frac{1}{2} \mu(B)^q ~.
\]
The assertion now follows from Corollary \ref{cor:equality-global}. 
\end{proof}

\begin{rem} \label{rem:translations}
Clearly, the example of the Lebesgue measure on $\Real^n$ shows that one cannot dispense with translations in Corollary \ref{cor:equality-global} and Corollary \ref{cor:equality-local} when $p=\infty$. To see this for $p \in (0,\infty)$, consider $w(x_1,\ldots,x_n) = (x_n)_+^{1/p}$, which is translation invariant in the first $n-1$ coordinate directions.
\end{rem}

\section{$q$-Complemented Concave measures: definitions and properties} \label{sec:Properties}

\begin{defn}
A Borel measure $\mu$ on $\Real^n$ is said to be \emph{$q$-complemented concave} (``$q$-CC"), and is said to satisfy a $q$-complemented Brunn--Minkowski inequality (``$q$-CBM"), $q \in [-\infty,+\infty]$, if:
\begin{multline} \label{def:q-CBM}
\forall \lambda \in (0,1) \;\;\;  \forall \text{ Borel sets $A,B \subset \Real^n$ such that } \mu(\Real^n \setminus A),\mu(\Real^n \setminus B) <\infty ~, \\
\mu_*(\Real^n \setminus (\lambda A + (1-\lambda) B)) \leq \brac{\lambda \mu(\Real^n \setminus A)^{q} + (1-\lambda) \mu(\Real^n \setminus B)^{q}}^{1/q} ~.
\end{multline}
Here $\mu_*$ denotes the inner measure induced by $\mu$, and we employ the usual convention for the cases $q \in \set{-\infty,0,\infty}$. The class of all $q$-CBM measures on $\Real^n$ is denoted by $\C_{q,n}$, and we set $\C_q := \cup_{n \in \Natural} \C_{q,n}$.

\end{defn} 

\begin{rem} \label{rem:analytic}
Note that when $A$ and $B$ are Borel sets, their Minkowski sum need not be Borel; however, it is always analytic and hence universally measurable \cite{KechrisBook}. 
However, Suslin showed that the complement of an analytic set is not analytic, unless the set is Borel. Consequently, we are forced to use the inner measure in our formulation above, a technical point which was not needed in Definition \ref{def:q-concave}.
\end{rem}

\begin{rem}
The fact that we do not require inequality (\ref{def:q-CBM}) to hold when $\mu(\Real^n \setminus A) = \infty$ or $\mu(\Real^n \setminus B) = \infty$ is not important when $q\ge 0$, but becomes crucial when $q < 0$. This is the analogue of the assumption $\mu(A) \mu(B) > 0$ in the definition of $q$-concave measures. Indeed, if $\mu$ is non-integrable at the origin, and if $B$ contains  the origin while $A$ and $\lambda A + (1-\lambda)B$ do not, then (\ref{def:q-CBM}) could never hold, since the right-hand side is finite while the left-hand side is not. 
\end{rem} 

Our motivation for the above definition stems from the following immediate analogue of Proposition \ref{prop:positive-p}:

\begin{prop} \label{prop:isop-negative-q}
Let $\mu$ denote a $q$-CC and $q$-homogeneous measure on $\Real^n$, with $q < 0$. Let $\norm{\cdot}_K$ denote an arbitrary semi-norm on $\Real^n$ with open unit ball $K$ so that $\mu(\Real^n \setminus K) < \infty$. Then the following isoperimetric inequality holds:
\[
\mu^{-}_{\norm{\cdot}_K}(C) \geq -\frac{1}{q} \mu(\Real^n \setminus K)^q  \mu(C)^{1-q} ~,
\]
for any Borel set $C$ with $\mu(C) < \infty$, and with equality for $C = \Real^n \setminus t K$, $t > 0$. In other words, 
complements of homothetic copies of $K$ are isoperimetric minimizers for $(\Real^n,\norm{\cdot}_K,\mu)$.
\end{prop}
\begin{proof}
As in the proof of Proposition \ref{prop:positive-p}, we obtain by scaling and homogeneity that:
\[
\mu_*(\Real^n \setminus (A + t B)) \leq \brac{\mu(\Real^n \setminus A)^q + t \mu(\Real^n \setminus B)^q}^{1/q} ~,
\]
for all Borel sets $A$ and $B$ so that $\mu(\Real^n \setminus A) , \mu(\Real^n \setminus B) < \infty$. When $B$ is open then so is $A+tB$, so there is no need to use the inner measure. Let $C$ be a Borel set with $\mu(C) < \infty$, and set $A := \Real^n \setminus C$. Then:
\begin{eqnarray*}
\mu^{-}_{\norm{\cdot}_K}(C) &= &  \liminf_{\eps \rightarrow 0}\frac{\mu((A+\eps K) \setminus A)}{\eps} = \liminf_{\eps \rightarrow 0} \frac{\mu(\Real^n \setminus A) - \mu(\Real^n \setminus (A + \eps K)) }{\eps} \\
& \geq &\liminf_{\eps \rightarrow 0} \frac{\mu(C) - \brac{\mu(C)^q + \eps \mu(\Real^n \setminus K)^q}^{1/q}}{\eps} = -\frac{1}{q}\mu(\Real^n \setminus K)^q \mu(C)^{1-q}  ~,
\end{eqnarray*}
with equality for $C = \Real^n \setminus K$, since $A + \eps K = (1+\eps) K$ when $K$ is convex. By scaling, we see that equality holds for complements of all positive homothetic copies of $K$. 
\end{proof}
\begin{rem}
As in Remark \ref{rem:no-convex}, observe that the convexity of $K$ is not required at all for the proof of the asserted inequality, but only for asserting equality for complements of homothetic copies of $K$. 
\end{rem} 

\subsection{Properties}

Recall that a sequence of measures $\set{\mu_k}$ on a common measurable space $(\Omega,\F)$ is said to converge strongly to $\mu$ (on $\F$) if for any measurable set $A \in \F$ we have $\lim_{k \rightarrow \infty} \mu_k(A) = \mu(A)$. Also, we say that $\mu$ is supported in $K \in \F$ if $\mu(\Omega \setminus K) = 0$.

\begin{prop} \label{prop:props}
\hfill
\begin{enumerate}
\item $\C_{q,n} \subset \C_{q',n}$ for all $-\infty \leq q \leq q' \leq +\infty$. 
\item $\C_{q,n}$ is closed with respect to (Borel) strong convergence. 
\item $\C_{q,n}$ is a cone, which is in addition convex if $q \leq 1$. In fact in the latter case,  if for some measure space $\left(\Omega,\mathcal{F}\right)$
  we have $\mu^\omega \in \C_{q,n}$ for every $\omega \in \Omega$,  $\eta$ is a $\sigma$-finite measure on $(\Omega,\F)$, and for every Borel set $A \subset \Real^n$, $\omega \mapsto \mu^\omega(A)$ is an $\F$-measurable function, then the Borel measure $\mu$ on $\Real^n$ given by: 
\[
\mu(A) := \int_{\Omega} \mu^{\omega}(A) d\eta(\omega) ~,
\]
satisfies $\mu \in \C_{q,n}$. 
\item If $T : \Real^n \rightarrow \Real^m$ is an affine map and $\mu \in C_{q,n}$ then $T_*(\mu) \in C_{q,m}$, where $T_*(\mu) = \mu \circ T^{-1}$ denotes the push-forward of $\mu$ by $T$.  In particular $C_{q}$ is closed under translations and taking marginals. 
\item Given a Borel measure $\mu$ on $\Real^n$ supported in a Borel set $K$, it is enough to use Borel sets $A,B \subset K$ when testing whether $\mu \in \C_{q,n}$ in (\ref{def:q-CBM}). 
\item If $\mu \in \C_{q,n}$, then also $\mu_C \in \C_{q,n}$ for any Borel subset $C \subset \Real^n$, where:
\begin{equation} \label{eq:mu-C}
\mu_C(A) := \begin{cases} \mu(A) & A \cap C = \emptyset \\ +\infty & \text{otherwise} \end{cases} ~.
\end{equation}
\end{enumerate}
\end{prop}
\begin{proof}
\begin{enumerate}
\item
The first assertion is immediate by Jensen's inequality. 
\item 
The second assertion is immediate since given $\lambda \in [0,1]$ and $A,B$ Borel subsets of $\Real^n$, the definition in (\ref{def:q-CBM}) only depends on the three values of $\mu(\Real^n \setminus A)$, $\mu(\Real^n \setminus B)$ and $\mu(\Real^n \setminus C)$ for any fixed Borel subset $C \supset (\lambda A + (1-\lambda) B))$, and so the assertion follows by the continuity of $\Real_+^2 \ni (a,b) \rightarrow (\lambda a^q + (1-\lambda) b^q)^{1/q}$ and the definition of strong convergence. 
\item
The cone property is obvious. When $q \leq 1$, the cone is in fact convex by the \emph{reverse} triangle inequality for non-negative functions in $L^q$ on the two point space $\set{1,2}$ with weights $\set{\lambda,1-\lambda}$, a space we denote by $L^q(\lambda,1-\lambda)$. Indeed, for any fixed $\lambda \in (0,1)$ and Borel sets $A,B \subset \Real^n$ such that $\Real^n \setminus A$ and $\Real^n \setminus B$ have finite $\mu$-measure (and hence by Fubini finite $\mu^\omega$-measure for $\eta$-almost every $\omega$), we have:
\begin{eqnarray}
\nonumber & & \mu_*(\Real^n \setminus (\lambda A + (1-\lambda) B)) \leq \int_{\Omega} \mu^\omega_*(\Real^n \setminus (\lambda A + (1-\lambda) B)) d\eta(\omega) \\
\label{eq:individual-CBM} &\leq &  \int_{\Omega} \brac{\lambda \mu^\omega(\Real^n \setminus A)^{q} + (1-\lambda) \mu^\omega(\Real^n \setminus B)^{q}}^{1/q} d\eta(\omega) \\
\nonumber & = & \int_{\Omega} \norm{ (\mu^\omega(\Real^n \setminus A),\mu^\omega(\Real^n \setminus B)) }_{L^q(\lambda,1-\lambda)} d\eta(\omega) \\
\label{eq:reverse-triangle} & \leq & \norm{ \brac{\int_{\Omega} \mu^\omega(\Real^n \setminus A) d\eta(\omega) , \int_{\Omega} \mu^\omega(\Real^n \setminus B) d\eta(\omega)} }_{L^q(\lambda,1-\lambda)} \\
 \nonumber & = & \brac{\lambda \mu(\Real^n \setminus A)^{q} + (1-\lambda) \mu(\Real^n \setminus B)^{q}}^{1/q} ~.
\end{eqnarray}
\item
This assertion follows since for an affine map $T$ we have $T^{-1}(\lambda A + (1-\lambda) B) = \lambda T^{-1}(A) + (1-\lambda) T^{-1}(B)$ and $T^{-1}(\Real^m \setminus A) = \Real^n \setminus T^{-1}(A)$. 
\item
Let $A,B$ be two Borel subsets of $\Real^n$ so that $\mu(\Real^n \setminus A),\mu(\Real^n \setminus B) < \infty$, and set $A' := A \cap K$ and $B' := B \cap K$. Since $A \setminus A'$ and $B \setminus B'$ have zero $\mu$-measure, it
follows that $\mu(\Real^n \setminus A') = \mu(\Real^n \setminus A) < \infty$ and $\mu(\Real^n \setminus B') = \mu(\Real^n \setminus B) < \infty$, 
and so by the assumption that (\ref{def:q-CBM}) is satisfied for $A',B' \subset K$, we have: 
\begin{multline}
\label{eq:reduction-to-K}
\mu_*(\Real^n \setminus (\lambda A + (1-\lambda) B)) \leq 
\mu_*(\Real^n \setminus (\lambda A' + (1-\lambda) B')) \\
\leq \brac{\lambda \mu(\Real^n \setminus A')^{q} + (1-\lambda) \mu(\Real^n \setminus B')^{q}}^{1/q} = \brac{\lambda \mu(\Real^n \setminus A)^{q} + (1-\lambda) \mu(\Real^n \setminus B)^{q}}^{1/q} ~.
\end{multline}
\item 
Given two Borel sets $A,B \subset \Real^n$ so that $\mu_C(\Real^n \setminus A),\mu_C(\Real^n \setminus B) < \infty$, clearly both sets must contain $C$. Consequently, $\lambda A + (1-\lambda) B$ also contains $C$ for any $\lambda \in [0,1]$, and hence the $\mu_C$-measures of complements of all sets involved coincide with their corresponding $\mu$-measures, which by assumption satisfy the $q$-CBM condition (\ref{def:q-CBM}). It follows that $\mu_C$ is also $q$-CC. 
\end{enumerate}
\end{proof}

Since the zero measure is trivially a $q$-CC measure for all $q \in [-\infty,\infty]$, we see by the last assertion of the previous proposition that, for any Borel subset $C \subset \Real^n$, so is the measure:
\begin{equation} \label{eq:trivial-q-CBM}
\mu^0_C(A) := \begin{cases} 0 & A \cap C = \emptyset \\ +\infty & \text{otherwise} \end{cases} ~.
\end{equation}
By appropriately choosing the set $C \neq \emptyset$, it is possible to construct many further examples of $q$-CC measures $\mu$ having the property that $\mu(A) = \infty$ if $A \cap C \neq \emptyset$; see Proposition \ref{prop:qcbm-non-local} (3). However, we will consider non $\sigma$-finite examples such as (\ref{eq:trivial-q-CBM}) as having pathological nature, and defer the construction of non-pathological $q$-CC measures to the next section. 

Several additional interesting properties of $q$-CC measures pertain to their finiteness and infiniteness:
\begin{prop} \label{prop:infinite}
\hfill
\begin{enumerate}
\item Any non-zero measure $\mu\in\C_{q,n}$, $q \in [-\infty,\infty)$, satisfies $\mu(\Real^n)=\infty$. 
\item The subclass of \emph{finite} measures in $\C_{\infty,n}$ coincides with the subclass of \emph{finite} $-\infty$-concave measures on $\Real^n$.
\item The subclass of \emph{infinite} $-\infty$-concave measures is a \emph{strict} subset of the subclass of \emph{infinite} $\infty$-CC measures. 
\end{enumerate}
\end{prop}
\begin{proof}
\hfill
\begin{enumerate}
\item 
Assume in the contrapositive that $\mu(\Real^n) < \infty$, and hence by rescaling we may assume that $\mu$ is a probability measure. By the first assertion of Proposition \ref{prop:props}, we know that $\mu$ is $q$-CC for some $q \in (0,\infty)$. The following argument mimics the proof of Borell's lemma from \cite{Borell-logconcave}. Note that by the triangle inequality, for any $s > 0$ and $t>1$:
\[
\frac{2}{t+1} (\Real^n \setminus B(t s)) + \frac{t-1}{t+1} B(s) \subset \Real^n \setminus B(s) ~, 
\]
where $B(r)$ denotes the (say closed) Euclidean ball of radius $r$ and center at the origin. The $q$-CBM property implies that:
\[
\!\!\!\!\!\!\!\! \mu(B(s)) \leq \mu \brac{\Real^n \setminus \brac{\frac{2}{t+1} (\Real^n \setminus B(t s)) + \frac{t-1}{t+1} B(s) }} \leq \brac{ \frac{2}{t+1} \mu(B(ts))^q + \frac{t-1}{t+1} \mu(\Real^n \setminus B(s))^q }^{1/q} ~.
\]
Since $q > 0$ and $\mu$ is a probability measure, this is equivalent to:
\[
\mu(B(t s)) \geq \brac{\frac{t+1}{2} \mu(B(s))^q - \frac{t-1}{2} (1 - \mu(B(s)))^q }^{1/q} ~.
\]
Choosing $s > 0$ large enough so that $\mu(B(s)) > 1/2$, we see that the right-hand side tends to infinity as $t \rightarrow \infty$, in contradiction to our assumption that $\mu$ is a probability measure. 
\item
The assertion follows since when $\mu(\Real^n) < \infty$, then the inequality:
\[
\mu_*(\Real^n \setminus (\lambda A + (1-\lambda) B)) \leq \max( \mu(\Real^n \setminus A) , \mu(\Real^n \setminus B)) ~,
\]
for all Borel subsets $A,B \subset \Real^n$ (we automatically have $\mu(\Real^n \setminus A) , \mu(\Real^n \setminus B) < \infty$), 
is equivalent to:
\[
\mu(\lambda A + (1-\lambda) B) \geq \min(\mu(A) , \mu(B)) ~.
\]
Note that there is no need to check whether $\mu(A) \mu(B) > 0$, since otherwise the inequality holds trivially. 
\item
To show the asserted inclusion, observe that if $\mu$ is a $-\infty$-concave measure, then by Borell's characterization in Theorem \ref{thm:BBL}, we know that $\mu$ is supported on an $m$-dimensional affine subspace $E$ of $\Real^n$, where it has a density $w(x)$ with respect to the Lebesgue measure $Leb_E$ on $E$ which is $-1/m$-concave. Truncating $w(x)$ by setting $w_k(x) := \min(w(x),k)) 1_{B(x_0,k)}$ for some fixed $x_0 \in E$, it is easy to check that $w_k$ remains $-1/m$-concave, and hence by Theorem \ref{thm:BBL} $\mu_k := w_k(x) dLeb_E(x)$ is $-\infty$-concave. Since $\mu_k$ is finite, the second assertion of the current proposition implies that $\mu_k$ is $\infty$-CC, and since $\mu_k$ tends to $\mu$ strongly as $k \rightarrow \infty$ by monotone convergence, it follows by the second assertion of Proposition \ref{prop:props} that $\mu$ is also $\infty$-CC on $E$. Proposition \ref{prop:props} (5) then implies that $\mu$ is $\infty$-CC on $\Real^n$. 

To show that the inclusion is strict,  set $C = \set{x_1,x_2}$ consisting of two distinct points in $\Real^n$ (or any other non-convex set for that matter). Now observe that $\mu^0_C$ defined in (\ref{eq:trivial-q-CBM}) is $\infty$-CC but it is not $-\infty$-concave, as seen by testing (\ref{eq:BBL}) with $A = \set{x_1}$, $B=\set{x_2}$ and $\lambda = 1/2$. 
\end{enumerate}
\end{proof}

The convex cone property described in Proposition \ref{prop:props} is on one hand extremely useful, but at the same time indicates that the class of $q$-CC measures ($q \leq 1$) is very different from the one of Borell--Brascamp--Lieb $p$-concave measures, which certainly do not share this property, unless $p \in [1/(n+1),1/n]$ (see Borell \cite[Theorem 4.1 and subsequent paragraph]{BorellConvexMeasures}). In addition, the next proposition demonstrates that contrary to the local characterization of the class of $p$-concave measures given by Theorem \ref{thm:BBL}, $q$-CC measures do not have any local characterization in the following sense:
\begin{prop}[Non-locality of the $q$-CBM property] \label{prop:qcbm-non-local}
Let $\mu \in \C_{q,n}$.  
\begin{enumerate}
\item
If $\mu$ is locally finite at $x_0 \in \Real^n$, then there exists an $\eps > 0$ so that $\mu|_{B(x_0,\eps)}\notin \C_{q',n} \setminus \set{0}$ for all $q' \in [-\infty,\infty)$. 
\item
If $d\mu = f(x) dx$, then for any $M,R>0$ and $x_0 \in \Real^n$, $\min(f(x),M) 1_{B(x_0,R)}(x) dx \notin \C_{q',n} \setminus \set{0}$ for all $q' \in [-\infty,\infty)$. 
\end{enumerate}
Furthermore:
\begin{enumerate}
\setcounter{enumi}{2}
\item 
If $\mu$ is any Borel measure on $\Real^n$, then for any $x_0 \in \Real^n$ and bounded neighborhood $\Real^n \setminus C$ of $x_0$, we have $\mu_C \in \C_{q,n}$ for all $q \in [-\infty,\infty]$ (where recall $\mu_C$ was defined in (\ref{eq:mu-C})). 
\end{enumerate}
\end{prop}
\begin{proof}
The first two assertions are immediate since a non-zero $q'$-CC measure, $q' \in [-\infty,\infty)$, is always infinite by Proposition \ref{prop:infinite} (1). The last assertion follows since if $A,B$ are two Borel subsets of $\Real^n$ whose complements have finite $\mu_C$-measure, then necessarily $A$ and $B$ contain $C$. But since $\Real^n \setminus C$ is bounded, we easily see that $\lambda C + (1-\lambda) C = \Real^n$ for all $\lambda \in (0,1)$, and hence $\mu_C(\Real^n \setminus (\lambda A + (1-\lambda) B)) = \mu_C(\emptyset) = 0$ and so $\mu_C$ is trivially $q$-CC for any $q \in [-\infty,\infty]$. 
\end{proof}

\section{Homogeneous Measures} \label{sec:Homogeneous}

So far we have discussed various properties of $q$-CC measures, but we have yet to produce a
single non-trivial ($\sigma$-finite) example of such measures. In this section we will show that such measures
exist in abundance. 

We begin with the following lemma, which will be superseded by the ensuing theorem. 

\begin{lem} \label{lem:conc-CC}
Let $\mu$ be a $q$-concave and $q$-homogeneous measure on $\Real^n$, $q \in (-\infty,1/n] \setminus \set{0}$. Then $\mu$ is also $q$-CC. 
\end{lem}

Using the Borell--Brascamp--Lieb Theorem \ref{thm:BBL} and Lemma \ref{lem:homogeneous}, we immediately deduce the following (essentially equivalent) corollary:
\begin{cor} \label{cor:conc-CC}
Let $w : \Real^n \rightarrow \Real_+$ be a $p$-concave and $p$-homogeneous density on $\Real^n$,  $p \in [-1/n,\infty] \setminus \set{0}$. Then $\mu = w(x) dx$ is $q$-CC, with $\frac{1}{q} = \frac{1}{p} + n$.
\end{cor}

\begin{proof}[Proof of Lemma \ref{lem:conc-CC}]
Since $\mu$ is $q$-concave, it is also $-\infty$-concave by Remark \ref{rem:gen-BBL}. It follows by Proposition \ref{prop:infinite} that $\mu$ is also $\infty$-CC. It remains to employ the homogeneity of $\mu$ in order to ``self-improve" the degree of complemented-concavity from $\infty$ to $q$. 

Indeed, let $A,B \subset \Real^n$ denote two Borel sets with $\mu(\Real^n \setminus A), \mu(\Real^n \setminus B) < \infty$, and let $\lambda \in (0,1)$. Assume first that $\mu(\Real^n \setminus A) \mu(\Real^n \setminus B) > 0$ and define:
\begin{align*}
A' &= \frac{ \lambda \mu(\Real^n \setminus A)^q + (1-\lambda)\mu(\Real^n \setminus B)^q}
	{\mu(\Real^n \setminus A)^q} A ~; \\	
B' &= \frac{ \lambda \mu(\Real^n \setminus A)^q + (1-\lambda)\mu(\Real^n \setminus B)^q}
	{\mu(\Real^n \setminus B)^q} B ~; \\
\lambda' &= \frac{\lambda \mu(\Real^n \setminus A)^q}
{ \lambda \mu(\Real^n \setminus A)^q + (1-\lambda)\mu(\Real^n \setminus B)^q} ~.
\end{align*}
Then by $\infty$-complemented concavity (\ref{def:q-CBM}) applied to $A'$,$B'$ and $\lambda' \in (0,1)$: 
\[
\mu_\ast( \Real^n \setminus (\lambda' A' + (1-\lambda')B') ) \le
\max\brac{ \mu(\Real^n \setminus A'), \mu(\Real^n \setminus B')}. 
\]
Plugging in the definitions of $A'$,$B'$,$\lambda'$, and using the fact that $\Real^n \setminus \delta C = \delta (\Real^n \setminus C)$ ($\delta \neq 0$) and that $\mu$ is 
$q$-homogeneous, we obtain:
\begin{equation} \label{eq:lem-required}
\mu_\ast( \Real^n \setminus (\lambda A + (1-\lambda)B) ) \le
\brac{ \lambda \mu(\Real^n \setminus A)^q + (1-\lambda) \mu(\Real^n \setminus B)^q }^\frac{1}{q} ~,
\end{equation}
which is the desired $q$-CBM inequality. Finally, when $\mu(\Real^n \setminus A) \mu(\Real^n \setminus B) = 0$, (\ref{eq:lem-required}) is satisfied 
by an approximation argument. Indeed, whenever $\mu$ is not identically zero, since $\mu$ has a non-trivial density on an $m$-dimensional affine subspace of $\Real^n$, then for any Borel set $C$ with $\mu(\Real^n \setminus C) < \infty$, we can find Borel sets $C_k \subset C$ so that $(0,\infty) \ni \mu(\Real^n \setminus C_k) \rightarrow \mu(\Real^n \setminus C)$ as $k \rightarrow \infty$, and therefore:
\begin{eqnarray*}
& & \mu_\ast( \Real^n \setminus (\lambda A + (1-\lambda)B) ) \leq \liminf_{k \rightarrow \infty} \mu_\ast( \Real^n \setminus (\lambda A_k + (1-\lambda)B_k) )  \\
&\leq & \lim_{k \rightarrow \infty} \brac{ \lambda \mu(\Real^n \setminus A_k)^q + (1-\lambda) \mu(\Real^n \setminus B_k)^q }^\frac{1}{q} = \brac{ \lambda \mu(\Real^n \setminus A)^q + (1-\lambda) \mu(\Real^n \setminus B)^q }^\frac{1}{q} ~.
\end{eqnarray*}
\end{proof}

The main result of this section asserts that the requirement in Lemma \ref{lem:conc-CC} that $\mu$ be $q$-concave, is redundant when $q \leq 1$:

\begin{thm} \label{thm:hom-CBM}
Fix $q \leq 1$, $q \neq 0$, and let $\mu$ be any $\sigma$-finite $q$-homogeneous Borel measure on $\Real^n$. Then $\mu$ is $q$-CC. 
\end{thm}

Again, using Lemma \ref{lem:homogeneous}, we immediately deduce the following corollary:

\begin{cor}
Let $w : \Real^n \setminus \set{0}\rightarrow \Real_+$ denote a $p$-homogeneous Borel density, $p \in (-\infty,-\frac{1}{n-1}] \cup (-\frac{1}{n},0) \cup (0,\infty]$. Then $\mu = w(x) dx$ is $q$-CC, where $\frac{1}{q} = \frac{1}{p} + n$. 
\end{cor}

\begin{proof}[Proof of Theorem \ref{thm:hom-CBM}]
Let us disintegrate $\mu$ into its one-dimensional radial components $\mu_\theta$ supported on $\Real_{\theta} := \Real_+ \theta$, $\theta \in S^{n-1}$, as follows:
\begin{equation} \label{eq:disintegrate0}
\mu = \int_{S^{n-1}} \mu_\theta d\eta(\theta)  ~,~  \mu_\theta = r^{1/q - 1} dLeb_{\Real_\theta}(r \theta) ~,~ 
\end{equation}
where $\eta$ is a $\sigma$-finite Borel measure on $S^{n-1}$. Indeed,  for sets of the form $(s,t] A$ where $(s,t] \subset (0,\infty)$ and $A \subset S^{n-1}$ is Borel, we easily verify by homogeneity that:
\[
\mu((s,t] A) = q (t^{1/q} - s^{1/q}) \eta(A) ~,~ \eta(A) := \frac{1}{q (2^{1/q} - 1)} \mu((1,2] A) ~.
\]
Since sets of the above form generate the entire Borel $\sigma$-algebra in $\Real^n$, the representation (\ref{eq:disintegrate0}) follows by Caratheodory's extension theorem. 

Set $\frac{1}{p} = \frac{1}{q}-1$. Since $\mu_\theta = 1_{\set{r > 0}} r^{1/p} dLeb_{\Real \theta}(r)$ has a $p$-homogeneous and $p$-concave density on $\Real \theta$, it follows by Corollary \ref{cor:conc-CC} and Proposition \ref{prop:props} (5) that $\mu_\theta$ is $q$-CC on $\Real^n$ for every $\theta \in S^{n-1}$. Recall that when $q \leq 1$ then the class of $q$-CC measures is a convex cone, and so by the extended formulation of Proposition \ref{prop:props} (3) and the representation (\ref{eq:disintegrate0}), it follows that $\mu$ is also $q$-CC, concluding the proof. 
\end{proof}

\begin{rem}
In particular, we conclude that the Lebesgue measure (in fact, restricted to an arbitrary measurable cone) is a $1/n$-CC measure, 
satisfying  (\ref{def:q-CBM}) with $q=1/n$. Using Theorem \ref{thm:hom-CBM} and the properties of Section \ref{sec:Properties}, one can also easily construct many 
additional $q$-CC measures which are not necessarily homogeneous. For example, if $\mu$ is any non-zero
$q$-homogeneous $\sigma$-finite measure with $q \leq 1$, and $x_1,x_2,\ldots,x_k \in \Real^n$ are distinct points ($k \geq 2$),
then the measure
\[
\widetilde{\mu}(A) = \sum_{i=1}^k \mu(A + x_i) 
\]
is also $q$-CC, but no longer homogeneous. By using several different homogeneous measures, or even
different values of $q_i \leq q$ (as $\C_{q_i,n} \subset \C_{q,n}$), one can create many further examples as well. 
An interesting question is what are the extremal rays of the convex cone of $q$-CC measures on $\Real^n$ for a given $q \leq 1$. 
\end{rem}

\subsection{Isoperimetric Inequality}

As a corollary of Theorem \ref{thm:hom-CBM} and Proposition \ref{prop:isop-negative-q}, one immediately obtains a 
sharp isoperimetric inequality for $q$-homogeneous measures:

\begin{cor}  \label{cor:hom-iso}
Let $\mu$ denote a $q$-homogeneous $\sigma$-finite Borel measure on $\Real^n$ with $q < 0$; in particular, the following applies if $\mu = w(x) dx$ and $w : \Real^n \setminus \set{0} \rightarrow \Real_+$ is a $p$-homogeneous Borel density with $p \in (-1/n,0)$. Let $\norm{\cdot}_K$ 
denote an arbitrary semi-norm on $\Real^n$ with open unit ball $K$ so that $\mu(\Real^n \setminus K) < \infty$. Then the following 
isoperimetric inequality holds:
\begin{equation} \label{eq:hom-iso}
\mu^{-}_{\norm{\cdot}_K}(C) \ge -\frac{1}{q} \mu(\Real^n \setminus K)^q  \mu(C)^{1-q} 
\end{equation}
for any Borel set $C$ with $\mu(C) < \infty$, and with equality for $C = \Real^n \setminus tK$,
 $t > 0$.
In other words, complements of homothetic copies of $K$ are isoperimetric
minimizers for $(\Real^n, \norm{\cdot}_K, \mu)$.
\end{cor}
\begin{rem}
We refer to the Introduction for remarks and references regarding previously known partial results. The range $p \leq -1/n$ is excluded above for good reason: it is known that even for Euclidean space endowed with the density $w(x) = \abs{x}^{1/p}$ for $p$ in that range, isoperimetric minimizers do not exist \cite[Prop 7.3]{DHHT-IsopInSectorsWithDensity}.
\end{rem}
The next subsection addresses the characterization of the equality cases in the $q$-CBM inequality and corresponding isoperimetric inequality.

\subsection{Equality Cases}

In this subsection, our analysis will be based on Dubuc's characterization in Theorem \ref{thm:Dubuc} of the equality case in the $q$-BM inequality. In the next section, we will present a different argument based on a $1$-dimensional analysis, which applies in a more general setting on one hand, but assumes that one of the sets is star-shaped on the other. 

We begin with a technical lemma required for the ensuing theorem, in which the superfluous technical assumptions will be removed. Although we will only apply this lemma in dimension $1$, we formulate it arbitrary dimension. 

\begin{lem} \label{lem:eq-conc}
Fix $q<0$ and let $\mu = w(x) dx$ be a $q$-concave and $q$-homogeneous measure on $\Real^n$ with density bounded above on $S^{n-1}$. 
Let $A$ and $B$ be two Borel sets with $\mu(\Real^n \setminus A) , \mu(\Real^n \setminus B)\in (0,\infty)$, and assume in addition that $B$ is open and contains the origin. If for some $\lambda \in (0,1)$ we have:
\[ 
\mu( \Real^n \setminus (\lambda A + (1-\lambda)B) ) =
\brac{ \lambda \mu(\Real^n \setminus A)^q + (1-\lambda) \mu(\Real^n \setminus B)^q }^\frac{1}{q},
\] 
then up to null-sets $A\cap \Sigma$ and $B \cap \Sigma$ are homothetic and convex, where 
$\Sigma$ is the convex cone on which $\mu$ is supported. 
\end{lem}

Note that since $B$ is open then all the relevant sets are Borel, so we do not need to worry about
measurability issues.

\begin{proof}
By the homogeneity of $\mu$, we may scale $B$ to a set $B' = \frac{1}{c} B$ so that $\mu(\Real^n \setminus B') = \mu(\Real^n \setminus A)$, 
and set\footnote{In the published version, there was an erroneous and unnecessary reduction to the case $\delta = 1/2$} $\delta := \lambda / (\lambda + (1-\lambda) c) \in (0,1)$. Employing the homogeneity and using the assumption, we have:
\begin{eqnarray*}
&  & \mu\brac{ \Real^n \setminus \brac{\delta A + (1-\delta) B'}}^q  = \frac{1}{\lambda + (1-\lambda) c} \cdot
\mu\brac{ \Real^n \setminus \brac{\lambda A + (1-\lambda) B}}^q  \\
& = & \frac{1}{\lambda + (1-\lambda) c} \brac{ \lambda \mu(\Real^n \setminus A)^q + (1-\lambda) \mu(\Real^n \setminus B)^q } \\
& = &  \delta \mu(\Real^n \setminus A)^q + (1-\delta) \mu(\Real^n \setminus B' )^q =\mu(\Real^n \setminus B')^q = \mu(\Real^n \setminus A)^q ~.
\end{eqnarray*}

We now approximate the measure $\mu$ is a way similar to that in the proof of Proposition \ref{prop:infinite} (3). Set $w_k := \min(w,k)$ and $\mu_k := w_k(x) dx$. Since $w$ is $p$-concave (with $\frac{1}{q} = \frac{1}{p} + n$), then so is $w_k$, and hence $\mu_k$ is $q$-concave for all $k \geq 0$.  Since $B'$ contains a neighbourhood of the origin and $w$ is bounded on $S^{n-1}$ and is homogeneous, we have $\mu(\Real^n \setminus B') = \mu_k(\Real^n \setminus B')$ for large enough $k$. Since 
$\mu( \Real^n \setminus A) < \infty$ the origin is in the closure of $A$, so $\delta A+(1-\delta)B'$ also
contains a neighborhood of the origin. It follows that for large enough $k$ we have:
\[
\mu_k\brac{ \Real^n \setminus \brac{\delta A+ (1-\delta) B'}} = 
\mu\brac{ \Real^n \setminus \brac{\delta A+ (1-\delta) B'}} = \mu(\Real^n \setminus B')
= \mu_k(\Real^n \setminus B').
\]
Since $\mu_k$ has finite mass ($\mu$ is integrable at infinity as $q < 0$), we may take the complement and deduce that $\mu_k(\delta A+(1-\delta) B') = \mu_k(B')$.

Now we wish to calculate $\mu_k(A)$. On the one hand:
\[
\mu_k(\Real^n \setminus A) \le \mu(\Real^n \setminus A) = \mu(\Real^n \setminus B')
= \mu_k(\Real^n \setminus B'), 
\]
so $\mu_k(A) \ge \mu_k(B')$. On the other hand $\mu_k$ is $q$-concave, so:
\[
\mu_k(B') = \mu_k\brac{\delta A+ (1-\delta) B'} \ge 
\brac{\delta \mu_k(A)^q + (1-\delta) \mu_k(B')^q}^\frac{1}{q}.
\]
Solving for $\mu_k(A)$ we see that $\mu_k(A) \le \mu_k(B')$, so all together we have 
$\mu_k(A) = \mu_k(B')$.

Now we are ready to invoke Dubuc's characterization of equality in the $q$-BM inequality. We know that:
\[
\mu_k\brac{\delta A+(1-\delta) B'} = \mu_k(A) = \mu_k(B') = 
\brac{\delta \mu_k(A)^q + (1-\delta) \mu_k(B')^q}^\frac{1}{q} > 0 ~.
\]
Since $\mu_k$ is $q$-concave and is supported on the same cone $\Sigma$ as $\mu$, 
it follows by Corollary \ref{cor:equality-global} that up to a null-set, $A \cap \Sigma$ coincides with $(m B + b) \cap \Sigma$ for some $m > 0$ and $b \in \Real^n$, and that $w(m x+ b) = c w(x)$ for almost all $x \in B$ and some $c > 0$. Since $B$ was assumed to contain a neighborhood of the origin and since $w$ is continuous in the interior of $\Sigma$ and $p$-homogeneous, it follows that $\lim_{x \rightarrow b} w(x) = \lim_{x \rightarrow 0} w(x) = +\infty$. Since the $w$ was also assumed bounded on $S^{n-1}$, we must have $b=0$.

Finally, the convexity up to null-sets of $A \cap \Sigma$ and $B \cap \Sigma$ follows by mimicking the proof of the second part of Corollary \ref{cor:equality-global}. 
\end{proof}

Now we can prove a stronger version of Lemma \ref{lem:eq-conc}, which does not require any concavity
assumptions:
\begin{thm} \label{thm:eq-hom}
Fix $q<0$ and let $\mu$ denote a $\sigma$-finite $q$-homogeneous measure supported on the cone $\Sigma$. Let $A,B$ be two
Borel sets with $\mu(\Real^n \setminus A) , \mu(\Real^n \setminus B) \in (0,\infty)$, and assume in addition that $B$ is an open set containing 
the origin. If for some $\lambda \in (0,1)$ we have: 
\begin{equation} \label{eq:eq-hom-needed}
\mu( \Real^n \setminus (\lambda A + (1-\lambda)B) ) =
\brac{ \lambda \mu(\Real^n \setminus A)^q + (1-\lambda) \mu(\Real^n \setminus B)^q }^\frac{1}{q},
\end{equation}
then $A\cap \Sigma$ and $B \cap \Sigma$ are homothetic up to a null-set. 
Furthermore, if $\Sigma$ is convex, then up to null-sets, so are $A \cap \Sigma$ and $B \cap \Sigma$. 
\end{thm}

\begin{proof}
Let us examine more carefully the proof of Theorem \ref{thm:hom-CBM}; we proceed with the same notation as there. The proof of Theorem \ref{thm:hom-CBM} relied on the convex cone and intrinsicness properties given by Proposition \ref{prop:props} (3) and (5), so to obtain equality in (\ref{eq:eq-hom-needed}) we must have equality in all of the sequence of inequalities appearing in the proofs of these properties.  In particular, for $\eta$-almost-every $\theta \in S^{n-1}$:
\begin{itemize}
\item 
Setting $\Real_\theta := \Real_+ \theta$, $A_\theta := A \cap \Real_\theta$ and $B_\theta := B \cap \Real_\theta$, by Fubini $\mu_\theta(\Real_\theta \setminus A_\theta), \mu_\theta(\Real_\theta \setminus B_\theta)\in (0,\infty)$. In addition, $B \cap \Real \theta$ is open and contains the origin.
\item
We must have equality in the $q$-CBM inequality for the $1$-dimensional measure $\mu_\theta$ in (\ref{eq:individual-CBM}). It follows by Lemma \ref{lem:eq-conc} that $A_\theta$ coincides with $m_\theta B_\theta$ up to a $1$-dimensional null-set, for $m_\theta > 0$. 
\item 
We must have equality in the inequality (\ref{eq:reduction-to-K}) reducing to the $1$-dimensional case, i.e.:
\begin{equation} \label{eq:radial-reduction}
(\mu_\theta)_*(\Real_\theta \setminus (\lambda A + (1-\lambda) B)) = (\mu_\theta)_*(\Real_\theta \setminus (\lambda A_\theta + (1-\lambda) B_\theta))  ~.
\end{equation}
\end{itemize}
In addition, we must have equality in the reverse triangle inequality in $L_q(\lambda,1-\lambda)$ employed in (\ref{eq:reverse-triangle}). Since $q < 0$, the function $\Real_+^2 \ni (x,y) \mapsto (\lambda x^q + (1-\lambda) y^q)^{1/q}$ is strictly convex in the normal direction to $(x,y)$, but linear in the radial direction. Consequently, for $\eta$-almost-every $\theta \in S^{n-1}$, the two-dimensional points $(\mu_\theta(\Real_\theta \setminus A_\theta),\mu_\theta(\Real_\theta \setminus B_\theta))$ must lay on a one-dimensional linear subspace. Since $A_\theta = m_\theta B_\theta$ up to a null-set for $\eta$-almost-all $\theta \in S^{n-1}$, and $\mu_\theta$ is $q$-homogeneous, it follows that $m_\theta$ is $\eta$-almost-everywhere equal to a constant $m > 0$. This implies that $A$ and $m B$ coincide up to zero $\mu$-measure, and in particular, that $A \cap \Sigma = m (B \cap \Sigma)$ up to an $n$-dimensional null-set.

Finally, when $\Sigma$ is convex, the convexity up to null-sets of $A \cap \Sigma$ and $B \cap \Sigma$ follows exactly as in the previous lemma, by mimicking the proof of the second part of Corollary \ref{cor:equality-global}.
\end{proof}

To conclude this section, we state a characterization of the equality case in the corresponding isoperimetric inequality. For isoperimetric minimizers which are complements of convex sets, one may employ the Bonnesen idea exactly as in Section \ref{sec:Positive-P} and reduce to the characterization of the equality case in the $q$-CBM inequality given in Theorem \ref{thm:eq-hom}.  Arbitrary minimizers will be handled by a slightly more delicate analysis developed in the next section for more general densities, and so we only quote the following particular case of the more general Theorem \ref{thm:n-D-OCC-eq} from the next section.

\begin{thm} \label{thm:hom-isop-eq}
Fix $q<0$, and let $\mu$ denote a $\sigma$-finite $q$-homogeneous measure on $\Real^n$ supported on the cone $\Sigma$. Let $B$ be any 
convex domain containing the origin with $\mu(\Real^n \setminus B) \in (0,\infty)$. Assume that for some Borel set $C$ with $\mu(C) \in (0,\infty)$ we have:
\[ 
\mu^{-}_{B}(C) = -\frac{1}{q} \mu(\Real^n \setminus B)^q  \mu(C)^{1-q}.
\]
Then up to a null-set, $C \cap \Sigma$ is homothetic to $\Sigma \setminus B$.
\end{thm}
\begin{proof}
Apply Theorem \ref{thm:n-D-OCC-eq} with $\varphi(t) = t^{\frac{1}{q} -1}$ and consult Remark \ref{rem:no-density-needed}. The assertion remains valid for any star-shaped domain $B$, to be defined in the next section; in that case, it follows from Theorem \ref{thm:n-D-OCC-eq} that  if $\Sigma$ is assumed convex, then up to null-sets, so are $B \cap \Sigma$ and $\Sigma \setminus C$.
\footnote{The published version contains a proof under the assumption, erroneously omitted, that $B$ and $\Real^n \setminus C$ are convex.}
\end{proof}

\begin{rem}
Note that contrary to the positively homogeneous case, we do not require any translations in the characterization of isoperimetric minimizers for the negatively homogeneous case. Loosely speaking, the reason is that the requirement $\mu(\Real^n \setminus B) < \infty$ automatically synchronizes $B$ and $\mu$ to have the same translation invariance, since $\mu$ is non-integrable at the origin.
\end{rem}

\section{Non-homogeneous Measures} \label{sec:NonHomogeneous}

When the density is no longer assumed homogeneous, we must be careful and distinguish between different types of possible concavity inequalities (cf. Borell \cite[Section 5]{BorellConvexMeasures} where a variant of $q$-concave measures is considered, in which the weights $\lambda$ and $1-\lambda$ in (\ref{def:q-concave}) are both replaced by $1$). The pertinent one for deducing isoperimetric inequalities is formulated in the following:

\begin{defn}
Fix a continuous strictly monotone function $\Phi : \Real_+ \rightarrow [0,\infty]$, and let $B$ be a Borel subset of a convex cone $\Sigma$. 
A Borel measure $\mu$ on $\Sigma$ is called \emph{$B$-one-sided $\Phi^{-1}$-complemented concave} (``$\mu$ is OCC"), and is said 
to satisfy a \emph{$B$-one-sided $\Phi^{-1}$-complemented-Brunn-Minkowski inequality} (``OCBM inequality"), 
if:
\[
 \mu(\Sigma \setminus (A + t B)) \leq \Phi( \Phi^{-1}(\mu(\Sigma \setminus A)) + t \Phi^{-1}(\mu(\Sigma \setminus B)) ) ~,
 \]
 for all $t \geq 0$ and Borel sets $A \subset \Sigma$ so that $\mu(\Sigma \setminus A) < \infty$. 
\end{defn}

\begin{rem}
In this section, $\Phi$ will always be strictly decreasing on its support $[0,z]$. Our convention is that $\Phi^{-1}(0) = z < \infty$ if the support is compact, and $\Phi^{-1}(0) = \infty$ otherwise. It follows that $\Phi^{-1} \circ \Phi = Id$ on $[0,z]$, and that $\Phi^{-1} \circ \Phi \leq Id$ in general. In any case $\Phi \circ \Phi^{-1} = Id$ since $\Phi$ is continuous.  
\end{rem}

Since we no longer assume that our measure is $q$-concave, we first need to prove the required $1$-D results directly (in contrast to the previous section, where we could simply invoke the results of Section \ref{sec:Positive-P}).

\subsection{Inequalities in dimension $1$}

\begin{prop}[OCBM Inequality in dimension $1$] \label{prop:1-D-OCC}
Assume that $\varphi : (0,\infty) \rightarrow \Real_+$ is non-increasing, lower semi-continuous, non-integrable at the origin and integrable at infinity, and set $\mu = \varphi(t) dt$ on  $\Real_+$. Denote $\Phi(t) := \int_t^\infty \varphi(s) ds$, and let $B = [0,b]$. Then:
\begin{enumerate}
\item $\mu^{-}_{B}(C) \geq b \cdot \varphi \circ \Phi^{-1} (\mu(C))$ for any Borel set $C \subset \Real_+$ with $\mu(C) < \infty$. 
\item $\mu$ is $B$-one-sided $\Phi^{-1}$-complemented-concave. 
\end{enumerate}
\end{prop}

\begin{rem}
Throughout this section we set $\varphi(\infty) = 0$. 
\end{rem}

\begin{proof}
Note that one can rescale $B$ without changing the validity of either assertions, so we assume that $b=1$. 
Let $C \subset \Real_+$ denote a Borel set with $\mu(C) < \infty$. Assume that $\mu^{-}_B(C) < \infty$, since otherwise the isoperimetric inequality holds trivially. Denote $A := \Real_+ \setminus C$, and set $A' := \cap_{\eps > 0} (A+ \eps B) \supset A$, the ``right closure of $A$", and $C' := \Real_+ \setminus A' \subset C$. Clearly, $A'$ is closed to the right, i.e. $A' \ni x_i \nearrow x$ implies $x \in A'$. Clearly $\mu(C') = \mu(C)$ since otherwise $\mu^{-}_B(C) = \infty$; and furthermore $\mu^{-}_B(C') = \mu^{+}_B(A') = \mu^{+}_{B}(A) = \mu^{-}_{B}(C)$ by an elementary inspection.

Now consider the point $t_0 = \inf C'$; we have $[0,t_0] \subset A'$.  If $t_0 = \infty$ this means that $C' = \emptyset$ and so $C$ has zero measure, so there is nothing to prove. Otherwise, observe that $\mu^{-}_B(C')  = \mu^{+}_B(A') \geq \varphi(t_0)$ : indeed, if $\varphi(t_0) = 0$ this holds trivially. Otherwise, the lower semi-continuity and monotonicity of $\varphi$ imply that $\varphi(t_0 + \eps_1) > 0$ for some $\eps_1 > 0$; since $t_0$ is a left accumulation point of $C'$, and since $\mu^{+}_{B}(A') < \infty$, there must be an $\eps_0 \in (0,\eps_1)$ so that $(t_0,t_0+\eps_0) \subset C'$, since otherwise we would have a disjoint union $\cup_{i=1}^\infty (a_i - \eps_i , a_i)  \subset C'$ with $a_i \searrow t_0$ and $\eps_i  > 0$, yielding $\mu^{+}_{B}(A') \geq \varphi(t_0 + \eps_1) \liminf_{\eps \rightarrow 0} \frac{1}{\eps} \sum_{i=1}^\infty \min(\eps_i , \eps) = \infty$, a contradiction; hence, $\mu^{-}_{B}(C')  = \mu^{+}_{B}(A') \geq \liminf_{\eps \rightarrow 0} \frac{1}{\eps} \int_{t_0}^{t_0+\eps} \varphi(s) ds = \varphi(t_0)$, as claimed. 

Lastly, since $C' \subset [t_0,\infty)$ and hence $\mu(C) = \mu(C') \leq \mu([t_0,\infty)) = \Phi(t_0)$, it follows when $\varphi(t_0) > 0$ that:
\[
\mu^{-}_{B}(C) = \mu^{-}_{B}(C') \geq \varphi(t_0) \geq \varphi \circ \Phi^{-1} \circ \Phi(t_0) \geq \varphi \circ \Phi^{-1}(\mu(C)) ~,
\]
where we have used again the monotonicity of $\varphi$ and the fact that $\Phi(t_0) > 0$ by lower semi-continuity and hence $\Phi^{-1} \circ \Phi(t_0) = t_0$. When $\varphi(t_0) = 0$, we have $\mu(C) \leq \Phi(t_0) = 0$, and hence $\mu^{-}_{B}(C) \geq 0 = \varphi \circ \Phi^{-1}(0) = \varphi \circ \Phi^{-1}(\mu(C))$ again by lower semi-continuity. 
The first assertion is thus proved.

The second assertion with $B=[0,1]$ follows by integrating the first one. Indeed, if $\mu(\Real_+ \setminus A) < \infty$, then:
\begin{equation} \label{eq:psi}
\Psi(t) := \Phi^{-1}(\mu(\Real_+ \setminus (A + t B))) ~,~ t \in \Real_+ ~,
\end{equation}
 is a continuous finite-valued function, and so by the fundamental theorem of calculus and Fatou's lemma:
\begin{eqnarray}
\nonumber & & \Psi(t) - \Psi(0)  = \lim_{\eps \rightarrow 0} \int_0^t \frac{\Psi(s+\eps) - \Psi(s)}{\eps} ds \geq
\int_0^t \liminf_{\eps \rightarrow 0} \frac{\Psi(s+\eps) - \Psi(s)}{\eps} ds \\
\label{eq:integrate-isop} &= &\int_0^t \frac{\mu^{-}_B(\Real_+ \setminus (A + s B))}{\varphi \circ \Phi^{-1} (\mu(\Real_+ \setminus (A + s B)))} ds \geq \int_0^t ds \geq t \cdot \Phi^{-1}(\mu(\Real_+ \setminus B)) ~.
\end{eqnarray}
Note that we have used that $B$ is convex to write $A + (s+\eps) B = (A + sB) + \eps B$ and that $\Phi^{-1} \circ \Phi \leq Id$. The second assertion follows since $\Phi$ is non-increasing and $\Phi \circ \Phi^{-1} = Id$. 
\end{proof}

\begin{rem}
The OCBM inequality is \emph{false} for general log-convex functions $\varphi$ and Borel sets $B$. This may be checked by taking for instance $\varphi(t)=\brac{e^t-1}^{-1}$, $A=[0,1]$, $B=[0,1]\cup[2,3]$ and $t=3$; we omit the tedious calculation. However, it is easy to check that the proof of the isoperimetric inequality remains valid for any compact $B \subset \Real_+$, if we set $b = Leb(B)$.
\end{rem}

\subsection{Equality in dimension $1$}

\begin{prop}[Equality in OCBM inequality in dimension $1$] \label{prop:1-D-OCC-eq}
With the same assumptions and notation as in Proposition \ref{prop:1-D-OCC}:
\begin{enumerate}
\item Assume  that $\mu^{-}_{B}(C) = b \cdot \varphi \circ \Phi^{-1} (\mu(C))$ for some Borel set $C \subset \Real_+$ with $\mu(C) < \infty$. Then up to a null-set, $C$ coincides with a (possibly empty) right half-line $[c,\infty)$.
\item Assume that:
\[
 \mu(\Real_+ \setminus (A + t_0 B)) = \Phi( \Phi^{-1}(\mu(\Real_+ \setminus A)) + t_0 \Phi^{-1}(\mu(\Real_+ \setminus B)) ) ~,
 \]
 for some $t_0 > 0$ and Borel set $A \subset \Real_+$ so that $\mu(\Real_+ \setminus A) < \infty$. Then up to a null-set, $A$ coincides with an interval $[0,a]$.
 \end{enumerate}
\end{prop}
\begin{proof}
We may assume by scaling that $B = [0,1]$. To prove the first assertion, let $C$ be a Borel set on which the isoperimetric inequality is attained, having $\mu$-measure $v \in [0,\infty)$. 
We continue with the notation used in the proof of Proposition \ref{prop:1-D-OCC}, recalling that $A := \Real_+ \setminus C$, $A' := \cap_{\eps > 0} (A+ \eps B) \supset A$ and $C' := \Real_+ \setminus A' \subset C$. Clearly $\mu(C') = \mu(C)$ since otherwise $\mu^{-}_{B}(C) = \infty$ and so $C$ cannot be a minimizer, and furthermore $\mu^{-}_{B}(C') = \mu^{+}_{B}(A') = \mu^{+}_{B}(A) = \mu^{-}_{B}(C)$. Consequently, $C'$ is also a minimizer of measure $v$.

As in the proof of the isoperimetric inequality, we set $t_0 := \inf C'$ so that $[0,t_0] \subset A'$. If $t_0 = \infty$ this means that $C' = \emptyset$ and so $v=0$ and $C$ is a null-set. Otherwise, as in the proof of Proposition \ref{prop:1-D-OCC}, we know that $\mu^{-}_{B}(C')  = \mu^{+}_{B}(A') \geq \varphi(t_0)$. 
Now, since $C'$ is a minimizer, it follows that $\varphi \circ \Phi^{-1}(v) \geq \varphi(t_0)$. And since $\varphi$ is non-increasing, it follows that $t_0 \geq s_0 := \Phi^{-1}(v)$. But $C' \subset [t_0,\infty) \subset [s_0,\infty)$, and both sets from either side have $\mu$-measure $v$. It follows that they must coincide up to null-measure, so $C'$ must be (up to null-measure) a right half-line. Since $C' \subset C$, $\mu(C') = \mu(C)$ and $\varphi$ is non-increasing, it follows that $C$ is also a right half-line up to null-measure, concluding the proof of the first assertion. 

As for the equality case in the OCBM inequality, let $A \subset \Real_+$ and $t_0 > 0$ be as above. Note that the proof of the second assertion of Proposition \ref{prop:1-D-OCC} actually gives $\Psi(t) - \Psi(s) \geq t-s$ for all $t \geq s \geq 0$, where $\Psi$ was defined in (\ref{eq:psi}). On the other hand, we are given that $\Psi(t_0) = \Phi^{-1} \circ \Phi (\Psi(0) + t_0) \leq \Psi(0) + t_0$. Consequently, we deduce that $\Psi(t) - \Psi(0) = t$ for all $t \in [0,t_0]$, and so taking the derivative at $t=0$ we obtain that:
\[
\mu^{-}_B(\Real_+ \setminus A) = \varphi \circ \Phi^{-1}( \Real_+ \setminus A) ~.
\]
By the first assertion it follows that $\Real_+ \setminus A$ coincides with a right half-line up to a null-set, concluding the proof of the second assertion. 
\end{proof}

\subsection{Inequalities in dimension $n$}

We now establish an $n$-dimensional $OCBM$ inequality for a certain class of measures. Contrary to the previous section, where we first established a $CBM$ inequality and then deduced its associated isoperimetric inequality, in this section we prefer to establish both types of inequalities independently, permitting us to handle below sets $B$ which are not necessarily convex. 

\begin{defn}[Star-Shaped Domain]
A Borel set $B \subset \Real^n$ will be called a star-shaped domain if there exists a measurable function $\rho_B  : S^{n-1} \rightarrow (0,\infty)$ so that:
\[
B = \set{ r \theta \; ; \; \theta \in S^{n-1} ~,~ r \in [0,\rho_B(\theta))} ~.
\]
We will denote $\norm{\theta}_B = 1 / \rho_B(\theta)$ for $\theta \in S^{n-1}$, and extend $\norm{\cdot}_B$ as a $1$-homogeneous function to the entire $\Real^n$. 
\end{defn}

\begin{thm} \label{thm:n-D-OCC}
Let $B$ denote a star-shaped domain in $\Real^n$ and let $w_0 : \Real^n \rightarrow \Real_+$ denote a $1$-homogeneous measurable function so that $\int_B w_0(x) dx < \infty$. Let $\varphi : \Real_+ \rightarrow \Real_+$ denote a non-increasing log-convex function, non-integrable at the origin and integrable at infinity. Set:
\[
\mu = w_0(x / \norm{x}_B) \norm{x}_B^{1-n} \varphi(\norm{x}_B) dx ~. 
\]
Then, denoting  $\Phi(t) := \int_t^\infty \varphi(s) ds$ and $I := \varphi \circ \Phi^{-1}$, we have:
\begin{enumerate}
\item 
\[
\mu^{-}_{B}(C) \geq I(\mu(C)) \;\;\; \text{for any Borel set $C \subset \Real^n$ with $\mu(C) < \infty$} ~.
\]
\item If $B$ is in addition convex, complements of homothetic copies of $B$ are isoperimetric minimizers for $(\Real^n,\norm{\cdot}_B,\mu)$.
\item $\mu$ is $B$-one-sided $\Phi^{-1}$-complemented-concave:
\[
 \mu_*(\Real^n\setminus (A + t B)) \leq \Phi( \Phi^{-1}(\mu(\Real^n \setminus A)) + t \Phi^{-1}(\mu(\Real^n \setminus B)) ) ~,
 \]
 for all $t \geq 0$ and any Borel set $A$ so that $\mu(\Real^n \setminus A) < \infty$.
\end{enumerate}
\end{thm}
\begin{proof}
Let us disintegrate $\mu$ into its one-dimensional radial components $\mu_\theta$ defined on $\Real_{\theta} := \Real_+ \theta$, $\theta \in S^{n-1}$, as follows:
\begin{equation} \label{eq:disintegrate}
\mu = \int_{S^{n-1}} \mu_\theta d\eta(\theta)  ~,~  
\end{equation}
where:
\[
\mu_\theta = \varphi_\theta(r)  dLeb_{\Real_{\theta}}(r \theta) ~,~ \varphi_\theta(r) = \frac{1}{\rho_B(\theta)} \varphi(r / \rho_B(\theta)) ) ~,~ d\eta(\theta)  := w_0(\rho_B(\theta) \theta) \rho_B(\theta)^n d\theta ~.
\]
Indeed, for any Borel set $C \subset \Real^n$:
\[
\mu(C) = \int_{S^{n-1}} \int_0^\infty 1_{C}(r \theta) r^{n-1} w_0(\rho_B(\theta) \theta ) \rho_B(\theta)^{n-1} r^{1-n} \varphi(r / \rho_B(\theta)) dr d\theta = \int_{S^{n-1}} \mu_\theta(C) d\eta(\theta) ~.
\]
By integrating in polar coordinates, it is immediate to verify that the assumption that $\int_B w_0(x) dx < \infty$ is equivalent to $\eta$ being a finite Borel measure on $S^{n-1}$. 
Since we may assume that $\eta$ is not identically zero (otherwise there is nothing to prove), by scaling $\varphi$ (which does not influence its log-convexity), we may also assume that $\eta(S^{n-1}) = 1$. It follows by Fubini's theorem that if $\mu(C) < \infty$ then $\mu_\theta(C_\theta) < \infty$ for $\eta$-almost every $\theta \in S^{n-1}$, where $C_\theta := C \cap \Real_\theta$. 

Now by Proposition \ref{prop:1-D-OCC}, we know that for any $C_\theta \subset \Real_\theta$ with $\mu_\theta(C_\theta) < \infty$ we have:
\[
(\mu_\theta)^{-}_{B_\theta}(C_\theta) \geq \rho_B(\theta) I_\theta(\mu_\theta(C_\theta)) = I(\mu_\theta(C_\theta)) ~,
\]
where $\Phi_\theta = \int_t^\infty \varphi_\theta(r) dr$ and $I_\theta = \varphi_\theta \circ (\Phi_\theta)^{-1} = \frac{1}{\rho_B(\theta)} I$.
In addition, observe that the log-convexity of $\varphi$ is equivalent to the convexity of $I = \varphi \circ \Phi^{-1}$; indeed, when $\varphi$ is positive and twice differentiable:
\begin{equation} \label{eq:I-convex}
I'  = (\log \varphi)' \circ \Phi^{-1} ~,~ I'' = \frac{(\log \varphi)''}{\varphi} \circ \Phi^{-1} ~,
\end{equation}
and the general case follows by approximation. In conjunction with Fatou's lemma, set inclusion and Jensen's inequality, it follows that:
\begin{eqnarray}
\nonumber & &  \mu^{-}_B(C) = 
\liminf_{\eps \rightarrow 0} \frac{1}{\eps} \int_{S^{n-1}} \mu_\theta(((\Real^n \setminus C) + \eps B) \setminus (\Real^n \setminus C)) d\eta(\theta)  \\
\label{eq:dummy-inq} &\geq& \int_{S^{n-1}} \liminf_{\eps \rightarrow 0} \frac{1}{\eps} \mu_\theta(((\Real^n \setminus C) + \eps B) \setminus (\Real^n \setminus C)) d\eta(\theta)  \\
\label{eq:1D-reduction} & \geq & \int_{S^{n-1}} \liminf_{\eps \rightarrow 0} \frac{1}{\eps} \mu_\theta(((\Real_\theta \setminus C) + \eps B_\theta) \setminus (\Real_\theta \setminus C)) d\eta(\theta)  \\
\label{eq:1D-theorem} & = & \int_{S^{n-1}} (\mu_\theta)^{-}_{B_\theta}(C_\theta) d\eta(\theta) \geq \int I(\mu_\theta(C_\theta)) d\eta(\theta) \\
\label{eq:Jensen} &\geq &  I(\int \mu_\theta(C_\theta) d\eta(\theta)) = I(\mu(C)) ~.
\end{eqnarray}

For the second assertion,  observe that when $C$ is the complement of a star-shaped domain in $\Real^n$, all inequalities above with the possible exception of (\ref{eq:1D-reduction}) and (\ref{eq:Jensen}) are in fact equalities. When $B$ is in addition assumed convex and $C$ is the complement of a homothetic copy of $B$, then we also have equality in (\ref{eq:1D-reduction}), and as $\theta \mapsto \mu_\theta(C_\theta)$ is constant on $S^{n-1}$, we have equality in Jensen's inequality (\ref{eq:Jensen}).

For the last assertion,
note that as usual $\mu_\theta(\Real^n \setminus A) < \infty$ for $\eta$-almost every $\theta \in S^{n-1}$ by Fubini's theorem. Applying Proposition \ref{prop:1-D-OCC}, we conclude that for $\eta$-almost every $\theta$ we have:
\begin{align*}
\brac{\mu_\theta}_\ast ( \Real_\theta \setminus (A_\theta + t B_\theta) ) & \le \Phi_\theta \brac{\Phi_\theta^{-1} \brac{\mu_\theta (\Real_\theta \setminus A_\theta)} + 
 t \Phi_\theta^{-1} \brac{\mu_\theta (\Real_\theta \setminus B_\theta) } } \\
&=  \Phi \brac{\Phi^{-1} \brac{\mu_\theta (\Real_\theta \setminus A_\theta)} + 
 t \Phi^{-1} \brac{\mu_\theta (\Real_\theta \setminus B_\theta) } } \\
&= \Phi \brac{\Phi^{-1} \brac{\mu_\theta (\Real_\theta \setminus A_\theta)} + t \Phi^{-1} \circ \Phi(1)  }.
\end{align*}
Here  $\Phi_\theta(t) = \int_t^\infty \varphi_\theta(r)dr = \Phi(t / \rho_B(\theta))$, and so the first equality above is immediate, and the second one follows since $\mu_\theta (\Real_\theta \setminus B_\theta)  = \Phi_\theta(\rho_B(\theta)) = \Phi(1)$. We set $t_0 := t \Phi^{-1} \circ \Phi(1) \in (0,t]$ (recall that $\Phi^{-1} \circ \Phi(1)$ may be strictly smaller than $1$ if $\Phi(1) = 0$). 

Define $F: \Real_+ \to \Real_+$ by $F(x) = \Phi( \Phi^{-1}(x)+t_0)$. When $\varphi$ is positive and differentiable, direct calculation confirms that: 
\[F''(x) = \frac{\varphi(a+t_0)}{\varphi^2(a)} ((\log \varphi)'(a) - (\log \varphi)'(a+t_0)) ~,~ a = \Phi^{-1}(x) ~.
\]In that case, the log-convexity of $\varphi$ implies that $F$ is concave; the general case follows by approximation. Consequently, by set inclusion and Jensen's inequality, we obtain:
\begin{eqnarray}
\nonumber & & \mu_\ast ( \Real^n \setminus (A+tB) ) \le \int_{S^{n-1}} \brac{\mu_\theta}_\ast ( \Real^n \setminus (A+t B)) d\eta(\theta) \\
\label{eq:sneaky0} &\le& \int_{S^{n-1}} \brac{\mu_\theta}_\ast ( \Real_\theta \setminus (A_\theta+t B_\theta)) d\eta(\theta) \\
\nonumber &\le & \int_{S^{n-1}} F\brac{\mu_\theta(\Real_\theta \setminus A_\theta) }d\eta(\theta) 
\le F\brac{\int_{S^{n-1}} \mu_\theta(\Real_\theta \setminus A_\theta) d\eta(\theta) } \\
\nonumber &=& F\brac{\int_{S^{n-1}} \mu_\theta(\Real^n \setminus A) d\eta(\theta) } = F(\mu(\Real^n \setminus A) ) = \Phi( \Phi^{-1}(\mu(\Real^n \setminus A))+t_0) \\
\nonumber &=& \Phi( \Phi^{-1}(\mu(\Real^n \setminus A))+t\Phi^{-1}(\mu(\Real^n \setminus B))) ~,
\end{eqnarray}
which is the desired assertion. The last equality follows since as explained above $\mu_\theta(\Real^n \setminus B) = \mu_\theta(\Real_\theta \setminus B_\theta)  = \Phi(1)$ for all $\theta \in S^{n-1}$ and $\eta$ was assumed to be a probability measure. 

A final remark: when $B$ is convex, the last assertion may also be obtained by integrating as in (\ref{eq:integrate-isop}) the infinitesimal isoperimetric form obtained in the first assertion, yielding the desired:
\[
\Phi^{-1}(\mu(\Real^n \setminus (A + tB))) - \Phi^{-1}(\mu(\Real^n \setminus A)) \geq t  \geq t \cdot \Phi^{-1}(\mu(\Real^n \setminus B)) ~.
\]
\end{proof}

\begin{rem} \label{rem:no-density-needed}
Note that all the results remain valid for any measure $\mu$ of the form (\ref{eq:disintegrate}), where $\eta$ is any finite Borel measure on $S^{n-1}$. The assumption in the formulation of the Theorem that $\eta$ has an integrable density is only for aesthetical convenience. 
\end{rem}

\begin{rem}
When $B$ is the Euclidean ball and $w_0$ and $\varphi$ are continuous and positive on $S^{n-1}$ and $(0,\infty)$, respectively, the first two assertions were proved by Howe \cite{Howe-IsoperimetryInWarpedProducts} by using a warped-product representation of the space. As Howe points out, his results and methods should apply for arbitrary star-shaped smooth compact domains $B$, but even in that case, it seems that Howe's warped-product boundary measure does not coincide with ours, and there are various technicalities which should be taken care of, such as using the convexity of $B$ where needed and defining the corresponding measure $\mu$ correctly; furthermore, it is not clear how to extend Howe's  method to the case when $w_0$ is no longer assumed continuous or positive, in which case the condition that $\int_B w_0(x) dx < \infty$ becomes meaningful. 
\end{rem}

\subsection{Characterization of equality in dimension $n$}

We will require the following intermediate results, which may be of independent interest:

\begin{defn}
The radial sum of two star-shaped domains $A,B \subset \Real^n$ is denoted $A+_r B$ and defined by:
\[
\rho_{A+_r B}(\theta) := \rho_A(\theta) + \rho_B(\theta) ~.
\]
\end{defn}

\begin{lem} \label{lem:sneaky}
Let $\Sigma \subset \Real^n$ denote a convex cone, and let $A,B \subset \Real^n$ denote two star-shaped domains with $Leb(A \cap \Sigma), Leb(B \cap \Sigma) < \infty$. Assume that:
\[
Leb \set{(A + B) \cap \Sigma) \setminus \brac{(A \cap \Sigma) +_r  (B \cap \Sigma)} } = 0 ~.
\]
Then up to null-sets, $A \cap \Sigma$ and $B \cap \Sigma$ are homothetic. Furthermore, if both have positive Lebesgue measure, then they are both convex up to null-sets. 
\end{lem}
\begin{proof}
Set $\Sigma_0 := \Sigma \cap S^{n-1}$. On one hand, by polar integration and the triangle inequality in $L^n$:
\begin{eqnarray*}
& & Leb((A \cap \Sigma) +_r  (B \cap \Sigma)) = \brac{\frac{1}{n} \int_{\Sigma_0} (\rho_A(\theta) +\rho_B(\theta))^n d\theta}^{1/n} \\
&\leq &  \brac{\frac{1}{n} \int_{\Sigma_0}\rho_A(\theta)^n d\theta}^{1/n} + \brac{\frac{1}{n} \int_{\Sigma_0}\rho_B(\theta)^n d\theta}^{1/n} =
Leb(A \cap \Sigma)^{1/n} + Leb(B \cap \Sigma)^{1/n} ~. 
\end{eqnarray*}
On the other hand, by the convexity of $\Sigma$ and the Brunn--Minkowski inequality:
\[
Leb((A + B) \cap \Sigma)^{1/n} \geq Leb((A \cap \Sigma) + (B \cap \Sigma))^{1/n} \geq Leb(A\cap \Sigma)^{1/n} + Leb(B \cap \Sigma)^{1/n} ~.
\]
It follows that all of the above inequalities are equalities. 

Now since $Leb(A \cap \Sigma), Leb(B \cap \Sigma) < \infty$ , it follows by the equality case in the $L^n$ triangle inequality that $\rho_A(\theta)$ and $\rho_B(\theta)$ must be co-linear for almost every $\theta \in \Sigma_0$, and therefore $A \cap \Sigma$ and $B \cap \Sigma$ are homothetic up to null-sets.

As for the convexity, this follows from the second part of Corollary \ref{cor:equality-global} applied to the measure $Leb|_{\Sigma}$. 
\end{proof}

\begin{lem} \label{lem:sneaky-isop}
Let $\Sigma \subset \Real^n$ denote a convex cone, and let $A,B \subset \Real^n$ denote two star-shaped domains with $Leb(A \cap \Sigma), Leb(B \cap \Sigma) < \infty$. Assume that:
\[
\liminf_{\eps \rightarrow 0} \frac{1}{\eps} Leb \brac{ (((A + \eps B) \setminus A) \cap \Sigma) \setminus (((A +_r \eps B) \setminus A) \cap \Sigma) }  = 0 ~.
\]
Then up to null-sets, $A \cap \Sigma$ and $B \cap \Sigma$ are homothetic. Furthermore, if both have positive Lebesgue measure, then they are both convex up to null-sets. 
\end{lem}
\begin{proof}
Set $\Sigma_0 := \Sigma \cap S^{n-1}$.
On one hand, by polar integration and H\"{o}lder's inequality:
\begin{eqnarray*}
\!\!\!\!\!\!\!\!\!\!\!\!\!\!\!\!\! & & \lim_{\eps \rightarrow 0}  \frac{1}{\eps} Leb(((A +_r \eps B) \setminus A) \cap \Sigma) = \lim_{\eps \rightarrow 0} \frac{1}{\eps} \frac{1}{n} \int_{\Sigma_0} \brac{(\rho_A(\theta) + \eps \rho_B(\theta))^n - \rho_A(\theta)^n} d\theta \\
\!\!\!\!\!\!\!\!\!\!\!\!\!\!\!\!\!  & = & \int_{\Sigma_0} \rho_A(\theta)^{n-1} \rho_B(\theta) d\theta \leq n \brac{\frac{1}{n} \int_{\Sigma_0} \rho_A(\theta)^n d\theta}^{\frac{n-1}{n}} \brac{\frac{1}{n} \int_{\Sigma_0} \rho_B(\theta)^n d\theta}^{\frac{1}{n}} = n Leb(A \cap \Sigma)^{\frac{n-1}{n}} Leb(B \cap \Sigma)^{\frac{1}{n}} ~;
\end{eqnarray*}
indeed, since $Leb(A \cap \Sigma), Leb(B \cap \Sigma) < \infty$ then we have uniform control over the lower-order terms in $\eps$, so the calculation of the limit above is well-justified. On the other hand, by the anisotropic isoperimetric inequality in the convex cone $\Sigma$ (Proposition \ref{prop:positive-p} and Remark \ref{rem:no-convex}):
\[
\liminf_{\eps \rightarrow 0} \frac{1}{\eps} Leb( ((A + \eps B) \setminus A) \cap \Sigma) = (Leb|_{\Sigma})^+_B(A) \geq n Leb(A \cap \Sigma)^{\frac{n-1}{n}} Leb(B \cap \Sigma)^{\frac{1}{n}}~.
\]
It follows that all of the above inequalities are equalities. 

Now since $Leb(A \cap \Sigma), Leb(B \cap \Sigma) < \infty$, it follows by the equality case in H\"{o}lder's inequality that $\rho_A(\theta)$ and $\rho_B(\theta)$ must be co-linear for almost every $\theta \in \Sigma_0$, and therefore $A \cap \Sigma$ and $B \cap \Sigma$ are homothetic up to null-sets. 

Finally, when both sets have positive Lebesgue measure, then by the equality case in the isoperimetric inequality for the measure $Leb|_{\Sigma}$ given by Corollary \ref{cor:equality-local}, it follows that $A \cap \Sigma$ and $B \cap \Sigma$ are up to null-sets convex. 
\end{proof}

\begin{rem}
We did not succeed in removing the seemingly superfluous assumption that $Leb(A \cap \Sigma), Leb(B \cap \Sigma) < \infty$ in the above lemmas, at least not for arbitrary star-shaped domains. \end{rem}

\begin{thm} \label{thm:n-D-OCC-eq}
With the notation and assumptions of Theorem \ref{thm:n-D-OCC}, let $\Sigma$ denote the cone on which $\mu$ is supported. Assume further that $\mu(B) > 0$ and that $\varphi$ is strictly positive. Then:
\begin{enumerate}
\item 
Assume that $\mu^{-}_{B}(C) = I(\mu(C)) $ for a Borel set $C$ with $\mu(C)  \in (0,\infty)$. If $\varphi$ is strictly log-convex or $Leb(B \cap \Sigma), Leb(\Sigma \setminus C) < \infty$, then $C \cap \Sigma$ is homothetic to $\Sigma \setminus  B$ up to a null-set. Furthermore, if $\Sigma$ is convex, then up to null-sets, so are $\Sigma \setminus C$ and $B \cap \Sigma$. 
\item 
Assume that:
\begin{equation} \label{eq:OCBM-equality-assumption}
 \mu(\Real^n\setminus (A + t_0 B)) = \Phi( \Phi^{-1}(\mu(\Real^n \setminus A)) + t_0 \Phi^{-1}(\mu(\Real^n \setminus B)) ) ~
 \end{equation}
 for some $t_0 > 0$ and a Borel set $A$ with $\mu(\Real^n \setminus A) \in (0,\infty)$. 
 If $\varphi$ is strictly log-convex or $Leb(B \cap \Sigma), Leb(A \cap \Sigma) < \infty$,
 then $A \cap \Sigma$ in homothetic to $B \cap \Sigma$ up to a null-set. Furthermore, if $\Sigma$ is convex, then up to null-sets, so are $A \cap \Sigma$ and $B \cap \Sigma$. 
\end{enumerate}
\end{thm}

\begin{proof}
We proceed with the same notation as in the proof of Theorem \ref{thm:n-D-OCC}. 
\begin{enumerate}
\item
 If $\mu^{-}_{B}(C) = I(\mu(C))$ with $\mu(C) \in (0,\infty)$, then we must have equality in the inequalities (\ref{eq:1D-theorem}) and (\ref{eq:Jensen}).
Equality in (\ref{eq:1D-theorem}) implies that for $\eta$-almost all $\theta \in S^{n-1}$ we have $\brac{\mu_\theta}_{B_\theta}^-(C_\theta) = I(\mu_\theta(C_\theta))$. As $\mu_\theta(C_\theta) < \infty$ and $\mu(B_\theta) > 0$ for $\eta$-almost all $\theta$, Proposition \ref{prop:1-D-OCC-eq} (1) implies that for $\eta$-almost all $\theta$, $C_\theta$ coincides with $[r_\theta,\infty)$ up to a null set, for some $r_\theta \in (0,\infty]$. For those $\theta \in S^{n-1}$:
\[
\mu_\theta(C_\theta) = \int_{r_\theta}^\infty \varphi_\theta(r)dr = 
\int_{r_\theta}^\infty \frac{1}{\rho_B(\theta)} \varphi\brac{\frac{r}{\rho_B(\theta)}}dr =
\Phi\brac{\frac{r_\theta}{\rho_B(\theta)}}.
\]

Now if $\varphi$ is assumed strictly log-convex, then by (\ref{eq:I-convex}) the function $I$ must be strictly convex. 
Hence the equality case of Jensen's inequality (\ref{eq:Jensen}) implies that $\mu_\theta(C_\theta)$ is constant  for $\eta$-almost all $\theta$.  Since $\Phi$ is strictly decreasing ($\varphi$ is strictly positive), it follows that $\frac{r_\theta}{\rho_B(\theta)}$ is constant for $\eta$-almost all $\theta \in S^{n-1}$. Consequently $C \cap \Sigma$ coincides with a homothetic copy of $(\Real^n \setminus B) \cap \Sigma$ up to a null set, as claimed.

In the case that $\Sigma$ is convex, we may alternatively argue as follows: denoting $A := \Real^n \setminus C$, recall that we also assume that $Leb(A \cap \Sigma) ,Leb(B \cap \Sigma) < \infty$. We know that $A_\theta$ coincides with $[0,r_\theta)$ up to a null-set for $\eta$-almost every $\theta$, and the assumption $Leb(A \cap \Sigma) < \infty$ implies that $r_\theta < \infty$ for $\eta$-almost every $\theta$. Consequently, without loss of generality, we may assume that $A$ is star-shaped (inspect the proof). 
Since we must have equality in all of the inequalities in the proof of Theorem \ref{thm:n-D-OCC} (1), it follows from the equality in (\ref{eq:dummy-inq}) and (\ref{eq:1D-reduction}) that:
\begin{equation} \label{eq:sneaky11}
\liminf_{\eps \rightarrow 0} \frac{1}{\eps} \mu \brac{ ((A + \eps B) \setminus A) \setminus ((A +_r \eps B) \setminus A) }  = 0 ~;
\end{equation}
a subtle point to note here is that indeed:
\[
\int_{S^{n-1}} \lim_{\eps \rightarrow 0} \frac{1}{\eps} \mu_\theta((A_\theta + \eps B_\theta) \setminus A_\theta ) d\eta(\theta)  = \lim_{\eps \rightarrow 0} \frac{1}{\eps} \int_{S^{n-1}} \mu_\theta((A_\theta + \eps B_\theta) \setminus A_\theta ) d\eta(\theta) 
\]
by monotone convergence, since $\varphi$ is non-increasing.
It is easy to verify that (\ref{eq:sneaky11}) implies the same with $\mu$ replaced by any measure $\nu \ll \mu$, and since $\varphi$ is assumed strictly positive, we may use $\nu = Leb|_{\Sigma}$. 
It follows by Lemma \ref{lem:sneaky-isop} that $A \cap \Sigma$ and $B \cap \Sigma$ are homothetic up to null-sets. As they both have positive $\mu$-measure and hence Lebesgue measure, they are also convex up to null-sets, and the assertion follows. 
\item
When $\varphi$ is strictly log-convex, the proof of the second assertion is exactly analogous to the above reasoning, 
relying on the characterization of equality in the 1-dimensional OCBM inequality given by Proposition \ref{prop:1-D-OCC-eq} (2) and the strict concavity of the function $F$ introduced in the proof of Theorem \ref{thm:n-D-OCC} (3). It follows that $A$ and $m B$ coincide up to zero $\mu$-measure for some $m>0$, and that in particular, $A \cap \Sigma$ coincides with $m B \cap \Sigma$ up to a null-set. 

In the case that $\Sigma$ is convex, we may alternatively argue as follows: since we must have equality in all of the inequalities in the proof of Theorem \ref{thm:n-D-OCC} (3), it follows from the equality in (\ref{eq:sneaky0}) and the fact that $\varphi$ is positive that $(A + t_0 B) \cap \Real_\theta$ coincides with $A_\theta + t_0 B_\theta$ up to a null-set, for $\eta$ almost every $\theta \in S^{n-1}$. Consequently, $(A+t_0 B) \cap \Sigma$ coincides with $(A \cap \Sigma) +_r  (t_0 B \cap \Sigma)$ up to a null-set; note that by the characterization of equality in the 1-dimensional OCBM inequality, we know that $A_\theta$ is an interval of the form $[0,\rho_A(\theta))$ up to a null-set for $\eta$ almost-every $\theta$, with $\rho_A(\theta) \in (0,\infty]$, and our assumption that $Leb(A \cap \Sigma) < \infty$ rules out the case that $\rho_A(\theta) = \infty$. Consequently, as in the first assertion, we may assume without loss of generality that $A$ is star-shaped. Applying Lemma \ref{lem:sneaky}, it follows that $A \cap \Sigma$ and $B \cap \Sigma$ are homothetic up to null-sets. Furthermore, since they have positive $\mu$-measure and hence Lebesgue measure, it follows that they are both also convex up to null-sets, concluding the proof. 
\end{enumerate}
\end{proof}

\begin{rem}
When $B$ is the Euclidean ball, $w_0$ and $\varphi$ are continuous and positive on $S^{n-1}$ and $(0,\infty)$, respectively (so $\Sigma = \Real^n$), and $\varphi$ is not required to be strictly log-convex, the first assertion was obtained by Howe \cite{Howe-IsoperimetryInWarpedProducts}. As our analysis indicates, various subtleties may occur when the positivity assumption on $w_0$ is removed. 
\end{rem}

\begin{rem}
We observe that most of our results for the $q$-CC class remain valid if we replace it by the class of $q$-Complemented-Radially-Concave measures $\mu$, obtained by replacing the Minkowski summation operation with radial summation, thereby increasing the left-hand side below and making the required inequality more restrictive:
\begin{multline*}
\forall \lambda \in (0,1) \;\;\;  \forall \text{ star-shaped domains $A,B \subset \Real^n$ such that } \mu(\Real^n \setminus A),\mu(\Real^n \setminus B) <\infty ~, \\
\mu(\Real^n \setminus (\lambda A +_r (1-\lambda) B)) \leq \brac{\lambda \mu(\Real^n \setminus A)^{q} + (1-\lambda) \mu(\Real^n \setminus B)^{q}}^{1/q} ~;
\end{multline*}
(and one may also extend this definition to all Borel sets). The reason is that many of our results were obtained by integrating a $1$-D result on rays, where both types of summation coincide. This connects our results to the ``Dual Brunn--Minkowski" Theory, initiated and developed by E. Lutwak \cite{Lutwak-dual-mixed-volumes} and subsequently many others (see \cite{GardnerGeometricTomography2ndEd} and the references therein), and so in this sense our results may be thought of as belonging to a ``Complemented Dual Brunn--Minkowski" Theory. 
\end{rem}

\section{Functional Inequalities} \label{sec:Functional}

\subsection{Functional Formulation of $q$-CBM}

The $q$-CBM property admits the following functional formulation:

\begin{prop}
The following statements are equivalent for a measure $\mu$ on $\Real^n$, $\lambda \in (0,1)$ and $q < 0$:
\begin{enumerate}
\item
For any Borel sets $A,B \subset \Real^n$ with $\mu(\Real^n \setminus A) , \mu(\Real^n \setminus B) < \infty$:
\[
\mu(\Real^n \setminus (\lambda A + (1-\lambda) B)) \leq \brac{\lambda \mu(\Real^n \setminus A)^q + (1-\lambda)\mu(\Real^n \setminus B)^q}^{1/q} ~.
\]
\item For any three measurable functions $f,g,h : \Real^n \rightarrow \Real_+$ satisfying:
\[
h(\lambda x + (1-\lambda) y) \leq \max(f(x),g(y)) \;\;\; \forall x,y \in \Real^n ~,
\]
if $\mu(\set{ f \geq t}) , \mu(\set{g \geq t}) < \infty$ for all $t > 0$, then we have:
\[
\int h d\mu \leq \brac{\lambda (\int f d\mu)^q + (1-\lambda)(\int g d\mu)^q}^{1/q} ~.
\]
\end{enumerate}
\end{prop}

\begin{proof}
The second assertion trivially implies first by using $f = 1_{\Real^n \setminus A}$, $g = 1_{\Real^n \setminus  B}$ and $h = 1_{\Real^n \setminus (\lambda A + (1-\lambda) B)}$. Now assume that the first assertion is satisfied, and observe that:
\[
\set{ h < t} \supset \lambda \set{f < t} + (1-\lambda) \set{g < t} ~.
\]
Consequently, our assumption implies that:
\[
\mu(\set{h \geq t}) \leq \brac{\lambda \mu(\set{f \geq t})^q + (1-\lambda)\mu(\set{g \geq t})^q}^{1/q} ~.
\]
And by the reverse triangle inequality in $L^{1/q}$ when $1/q \leq 1$, we obtain:
\[
\brac{\int h d\mu}^q = \brac{\int_0^\infty \mu(\set{h \geq t}) dt}^q \geq \brac{\int_0^\infty \brac{ \; \lambda \mu(\set{f \geq t})^q + (1-\lambda)\mu(\set{g \geq t})^q}^{1/q} dt}^{q}
\]
\[
\geq \lambda \brac{\int_0^\infty \mu(\set{f \geq t}) dt}^{q} + (1-\lambda) \brac{\int_0^\infty \mu(\set{g \geq t}) dt}^{q} = \lambda \brac{\int f d\mu}^q + (1-\lambda)\brac{\int g d\mu}^q ~,
\]
which is the desired assertion (since $q < 0$). 
\end{proof}

\begin{rem}
The analogous functional formulation for the class of $q$-concave measures would be that for every $q\in [-\infty,\frac{1}{n}]$,  $\mu$ is $q$-concave iff for all $\lambda \in (0,1)$, if:
\[
h(\lambda x + (1-\lambda) y) \geq \min(f(x),g(y)) \;\;\; \forall x,y \in \Real^n ~,
\]
then:
\[
\int h d\mu \geq \brac{\lambda (\int f d\mu)^q + (1-\lambda)(\int g d\mu)^q}^{1/q} ~;
\]
this is a straightforward consequence of Borell's characterization of $q$-concave measures in Theorem \ref{thm:BBL} and the general Borell--Brascamp--Lieb Theorem \ref{thm:Dubuc} (1).
\end{rem}

\subsection{Sobolev inequalities}

Using nowadays standard methods (see e.g. \cite{BobkovHoudreMemoirs,BobkovConvexHeavyTailedMeasures} and the references therein), we may rewrite the isoperimetric inequalities obtained in the previous sections as Sobolev-type inequalities. For the reader's convenience, we sketch the proofs. We begin with some definitions:

\begin{defn}
Let $I : \Real_+ \rightarrow \Real_+$ denote a function increasing from $I(0)=0$ to $I(\infty)=\infty$, and set $\Phi = I^{-1}$. Let $\mu$ denote a Borel measure on $\Real^n$. Given a measurable function $f : \Real^n \rightarrow \Real$, we define the following weak ``quasi-norms":
\[
\norm{f}_{L^{\Phi,1}(\mu)} := \int_0^\infty I(\mu(\set{\abs{f} \geq t})) dt ~;
\]
\[
\norm{f}_{L^{\Phi,\infty}(\mu)} := \sup_{t > 0} \; t \cdot I(\mu(\set{\abs{f} \geq t})) ~.
\]
The primary example is when $I(x) = x^{1/\alpha}$ and $\Phi(x) = \Phi_\alpha(x) = x^\alpha$, $\alpha > 0$, in which case we obtain (up to numeric constants) the usual Lorentz quasi-norms $L^{\alpha,1}(\mu)$ and $L^{\alpha,\infty}(\mu)$, respectively. We will use our normalization, and set $L^{\alpha,s}(\mu) = L^{\Phi_\alpha,s}(\mu)$, $s \in \set{1,\infty}$. 
\end{defn}
\begin{rem}
We use the term ``quasi-norm" above loosely, without claiming any mathematical meaning. When $I$ is concave (i.e, $\Phi$ is convex), it is easy to verify that the resulting homogeneous functionals are indeed quasi-norms, satisfying:
\[
\norm{f+g}_{L^{\Phi,s}(\mu)} \leq 2^{1/\alpha} \brac{\norm{f}_{L^{\Phi,s}(\mu)} + \norm{g}_{L^{\Phi,s}(\mu)}} ~,~ s \in \set{1,\infty} ~,
\]
with $\alpha = 1$. Similarly, when $\Phi(x) = x^\alpha$ and $\alpha \in (0,1)$ the above inequality is equally valid. However, when $I$ is a general convex function ($\Phi$ is a general concave one), no weak triangle inequality as above can be guaranteed. 
\end{rem}

\begin{defn}
Given a locally Lipschitz function $f : \Real^n \rightarrow \Real$ and a norm $\norm{\cdot}$ on $\Real^n$, we define the following measurable function on $\Real^n$ (cf. \cite{BobkovHoudreMemoirs}):
\[
\norm{\nabla f}^*(x) := \limsup_{y \rightarrow x} \frac{\abs{f(y) - f(x)}}{\norm{y-x}} ~.
\]
\end{defn}

\begin{prop} \label{prop:func-equiv}
Let $\mu$ denote a Borel measure on $\Real^n$ and let $\norm{\cdot}$ denote a norm on $\Real^n$. Given $I : \Real_+ \rightarrow \Real_+$ an increasing function from $I(0) = 0$ to $I(\infty) = \infty$, set $\Phi = I^{-1}$. The following statements are equivalent:
\begin{enumerate}
\item
For every Borel set $C \subset \Real^n$ with $\mu(C) < \infty$:
\[
\mu^-_{\norm{\cdot}}(C) \geq I(\mu(C)) ~.
\]
\item
For any locally Lipschitz function $f : \Real^n \rightarrow \Real$ so that $\mu(\set{ \abs{f} \geq t} ) < \infty$ for all $t > 0$, we have:
\[
\int \norm{\nabla f}^* d\mu \geq \norm{f}_{L^{\Phi,1}(\mu)} ~.
\]
\item
For any locally Lipschitz function $f : \Real^n \rightarrow \Real$ so that $\mu(\set{ \abs{f} \geq t} ) < \infty$ for all $t > 0$, we have:
\[
\int \norm{\nabla f}^* d\mu \geq \norm{f}_{L^{\Phi,\infty}(\mu)} ~.
\]
\end{enumerate}
\end{prop}
\begin{proof}[Sketch of Proof]
Clearly $\norm{f}_{L^{\Phi,1}(\mu)} \geq \norm{f}_{L^{\Phi,\infty}(\mu)}$, so statement (2) implies (3). To see that statement (3) implies (1), let $C \subset \Real^n$ denote a Borel set with $\mu(C) < \infty$ and $\mu^{-}_{\norm{\cdot}}(C) < \infty$ (otherwise there is nothing to prove). By applying (3) to functions of the form $f_\eps(x) = \max(1 - \inf_{y \in C} \norm{x-y} / \eps,0)$ and letting $\eps \rightarrow 0$, one recovers (1) (see \cite[Section 3]{BobkovHoudreMemoirs} for more details). Finally, (1) implies (2) by the generalized co-area inequality of Bobkov--Houdr\'e \cite{BobkovHoudreMemoirs}. Indeed:
\begin{eqnarray*}
& & \int \norm{\nabla f}^* d\mu   \geq \int \norm{\nabla \abs{f}}^* d\mu   \geq \int_0^\infty \mu^+_{\norm{\cdot}}(\set{ \abs{f} < t}) dt \\
&=& \int_0^\infty \mu^-_{\norm{\cdot}}(\set{ \abs{f} \geq t}) dt \geq  \int_0^\infty I(\mu(\set{ \abs{f} \geq t})) = \norm{f}_{L^{\Phi,1}(\mu)} ~.
\end{eqnarray*}
\end{proof}

Using Proposition \ref{prop:func-equiv}, we may clearly reformulate the isoperimetric inequalities derived in this work in functional Sobolev-type form. To avoid extraneous generality, we will only write this out explicitly in the homogeneous case. 

\begin{cor} \label{cor:func}
Let $w : \Real^n \setminus \set{0} \rightarrow \Real_+$ denote a locally-integrable $p$-homogeneous Borel density with $p \in (-1/n,0)$, and set $\mu = w(x) dx$. Let $\norm{\cdot}$ denote a norm on $\Real^n$ and set:
\[
 \frac{1}{q} = \frac{1}{p} + n ~,~ \alpha := \frac{1}{1-q} ~,~  C_1 := -\frac{1}{q}  \mu(\set{\norm{x} \geq 1})^q ~,~ C_2 := C_1^{-\alpha} ~.
 \]
 Then:
\begin{enumerate}
\item
For any locally Lipschitz function $f : \Real^n \rightarrow \Real$ with $f(0) = 0$ we have:
\[
\int \norm{\nabla f}^* d\mu \geq C_1 \norm{f}_{L^{\alpha,1}(\mu)} ~.
\]
\item
For any locally Lipschitz function $f : \Real^n \rightarrow \Real$ with $f(0) = 0$ and $\beta \in (0,\alpha)$, we have:
\[
\int \norm{\nabla f}^* d\mu \geq C_1\brac{\frac{\alpha-\beta}{\alpha}}^{1/\beta} \norm{f}_{L^{\beta}(\mu)} ~.
\]
\item
For any essentially bounded locally Lipschitz function $f : \Real^n \rightarrow \Real$ with $f(0) = 0$, we have:
\[
\norm{f}_{L^1(\mu)} \leq C_2  (\int \norm{\nabla f}^* d\mu)^{\alpha}  \norm{f}_{L^\infty(\mu)}^{1-\alpha} ~.
\]
\end{enumerate}
\end{cor}

\begin{rem}
When $\alpha \geq 1$, it is easy to check that:
\[
\norm{f}_{L^{\alpha,1}(\mu)} \geq \norm{f}_{L^\alpha(\mu)} = \brac{\int_0^\infty \alpha t^{\alpha-1} \mu(\set{ \abs{f} \geq t}) dt}^{1/\alpha} \geq \norm{f}_{L^{\alpha,\infty}(\mu)} ~,
\]
and consequently Proposition \ref{prop:func-equiv} applied with $I(x) = c x^{1/\alpha}$ yields the well known equivalence between an isoperimetric inequality of the form $\mu^-_{\norm{\cdot}}(A) \geq c \mu(A)^{1/\alpha}$ or $\mu^+_{\norm{\cdot}}(A) \geq c \mu(A)^{1/\alpha}$ (the difference is inconsequential) and a Sobolev inequality $\int \norm{\nabla f}^* d\mu \geq c\norm{f}_{L^{\alpha}(\mu)}$, as first noted by  Federer-Fleming \cite{FedererFleming} and Maz'ya \cite{MazyaSobolevImbedding} in connection to the optimal constant in Gagliardo's inequality on $\Real^n$.
However, when $\alpha \in (0,1)$, which is the relevant case for us here, we have:\[
\norm{f}_{L^\alpha(\mu)}  \geq \norm{f}_{L^{\alpha,1}(\mu)} ~,
\]
and so we cannot replace the $L^{\alpha,1}$ quasi-norm by the stronger $L^\alpha$ one in Proposition \ref{prop:func-equiv} or Corollary \ref{cor:func}. Variants of cases (2) and (3) above were also obtained by Bobkov \cite[Lemma 8.3]{BobkovConvexHeavyTailedMeasures} with the same $\alpha = \frac{1}{1-q}$ for the class of $q$-concave \emph{probability} measures, $q \in (-\infty,1/n]$, under the assumption that the function $f$ has zero median. In our results, this assumption on $f$ is replaced by the requirement that $f(0) = 0$, since our (\emph{non-probability}) measures necessarily explode at the origin. However, it is a very difficult problem to obtain good constants $C_1,C_2$ in Bobkov's isoperimetric problem (see \cite[Corollary 8.2]{BobkovConvexHeavyTailedMeasures}), in contrast to our sharp constants $C_1,C_2$ above. 
\end{rem}

\begin{proof}[Proof of Corollary \ref{cor:func}]
Note that $\mu$ is $q$-homogeneous with $q < 0$, and since its density is locally-integrable outside the origin, it follows that the complement of any neighborhood of the origin has finite $\mu$-measure. Consequently, $\mu(\set{\abs{f} \geq t}) < \infty$ for any continuous function $f$ with $f(0) = 0$ and $t > 0$, and so we may invoke the equivalent versions given by Proposition \ref{prop:func-equiv} to the isoperimetric inequality guaranteed by Corollary \ref{cor:hom-iso}:
\[
\mu^{-}_{\norm{\cdot}}(C) \geq C_1 \mu(C)^{\frac{1}{\alpha}} \;\;\; \forall \text{ Borel set $C \subset \Real^n$ with } \mu(C) < \infty ~.
\]
This yields the first statement. The second follows by passing to the equivalent Sobolev inequality with the $L^{\alpha,\infty}(\mu)$ norm, and estimating:
\[
\norm{f}_{L^\beta(\mu)}^\beta = \int_{0}^\infty \beta t^{\beta-1} \mu(\set{\abs{f} \geq t}) dt \leq s^\beta + \int_{s}^\infty \beta t^{\beta-1-\alpha} dt \cdot \norm{f}^\alpha_{L^{\alpha,\infty}(\mu)} ~,
\]
using the optimal $s = \norm{f}_{L^{\alpha,\infty}}$. 
The third statement is a Nash-type interpolation inequality which follows from the first statement and H\"{o}lder's inequality:
\[
\norm{f}_{L^1(\mu)} \leq (\norm{f}_{L^{\alpha,1}(\mu)})^{\alpha} (\norm{f}_{L^\infty(\mu)})^{1-\alpha} ~.
\]
\end{proof}
 
As usual, by applying the above inequalities to $f = g^r$ and invoking H\"{o}lder's inequality (see e.g. \cite[Section 2]{EMilman-RoleOfConvexity}), we can also pass to $L^r$-Sobolev inequalities with $r > 1$; we leave the precise formulation to the interested reader.

\bibliographystyle{plain}
\bibliography{../ConvexBib}

\end{document}